\documentclass[12pt,reqno]{amsart}
\usepackage{amsmath,amssymb,bbm,booktabs,esint,times,url}

\usepackage{tikz}
\usepackage[bookmarks=false,unicode,colorlinks,urlcolor=red,citecolor=red,linkcolor=blue,
]{hyperref}
\usepackage{aliascnt}
\usetikzlibrary{calc}

\usepackage[margin=1.25in]{geometry}

\newtheorem{theorem}{Theorem}[section]

\newaliascnt{lemma}{theorem}
\newtheorem{lemma}[lemma]{Lemma}
\aliascntresetthe{lemma}

\newaliascnt{proposition}{theorem}
\newtheorem{proposition}[proposition]{Proposition}
\aliascntresetthe{proposition}

\newaliascnt{corollary}{theorem}
\newtheorem{corollary}[corollary]{Corollary}
\aliascntresetthe{corollary}

\newaliascnt{conjecture}{theorem}

\aliascntresetthe{conjecture}

\newcommand{\B}{{\mathbb B}}
\newcommand{\D}{{\mathbb D}}
\newcommand{\e}{\varepsilon}
\newcommand{\Id}{\operatorname{Id}}
\newcommand{\Ray}{\operatorname{Ray}}
\newcommand{\R}{{\mathbb R}}
\newcommand{\Rd}{{\R^d}}
\newcommand{\Robin}{{\text{Robin}}}
\newcommand{\Sp}{{\mathbb S}}
\newcommand{\tr}{\operatorname{tr}}

\begin{document}

\title[]{Sharp spectral bounds on starlike domains}

\author[]{R. S. Laugesen and B. A. Siudeja}
\address{Department of Mathematics, Univ.\ of Illinois, Urbana,
IL 61801, U.S.A.}
\email{Laugesen\@@illinois.edu}
\address{Department of Mathematics, Univ.\ of Oregon, Eugene,
OR 97403, U.S.A.}
\email{Siudeja\@@uoregon.edu}
\date{\today}

\keywords{Isoperimetric, membrane, convex, spectral zeta, heat trace, partition function,sloshing.}
\subjclass[2010]{\text{Primary 35P15. Secondary 35J20,52A40}}

\begin{abstract}
We prove sharp bounds on eigenvalues of the Laplacian that complement the Faber--Krahn and Luttinger inequalities. In particular, we prove that the ball maximizes the first eigenvalue and minimizes the spectral zeta function and heat trace. The normalization on the domain incorporates volume and a computable geometric factor that measures the deviation of the domain from roundness, in terms of moment of inertia and a support functional introduced by P\'{o}lya and Szeg\H{o}.

Additional functionals handled by our method include finite sums and products of eigenvalues. The results hold on convex and starlike domains, and for Dirichlet, Neumann or Robin boundary conditions.
\end{abstract}

\maketitle

\vspace*{-12pt}

\section{\bf Introduction}
\label{sec:intro}

How do eigenvalues of the Laplacian depend on the shape of the domain? We will obtain new  quantitative estimates on the eigenvalues in terms of explicitly computable geometric functionals. 

Write $\lambda_j$ for the Dirichlet eigenvalues of the Laplacian on the bounded domain $\Omega$ in $\Rd, d \geq 2$, with corresponding $L^2$-orthonormal eigenfunctions $u_j$, so that
\[
\begin{cases}
-\Delta u_j = \lambda_j u_j \;\; \text{in $\Omega$} \\
\hfill u_j = 0 \;\; \text{on $\partial \Omega$}
\end{cases}
\]
and
\[
0 < \lambda_1 < \lambda_2 \leq \lambda_3 \leq \dots .
\]
These eigenvalues represent physical quantities such as frequencies of vibration, rates of decay to equilibrium in diffusion, and energy levels of quantum particles. 

The problem of understanding how eigenvalues are affected by the shape of the domain, and of identifying domains that extremize eigenvalues, is long-standing and difficult. Numerous monographs and survey articles summarize the state of knowledge in euclidean space \cite{AB07,B80,BL08,He06,K85,K06}. Important results have been obtained on closed surfaces too, for example, see \cite{JLNNP05,SY94}.

Let us first describe our main result in the special case of $2$ dimensions (\autoref{Deigen2}). Later we extend to all dimensions (\autoref{Deigen}) and to Neumann and Robin analogues (\autoref{Neigen} and \autoref{Reigen}), and finally to the sloshing eigenvalues (\autoref{sec:sloshing}). 

Consider a starlike domain as in \autoref{fig:starlike} and define scale-invariant geometric factors
\begin{equation} \label{Gin2dim}
G_0 = \frac{1}{2\pi} \int_{\partial \Omega} \frac{1}{x \cdot N(x)} \, ds(x) , \qquad
G_1 = \frac{2\pi I_\text{origin}}{A^2} ,
\end{equation}
where $N(x)$ is the outward unit normal vector, $A$ is the area of $\Omega$, and $I_\text{origin}=\int_\Omega |x|^2 \, dA$ is the polar moment of inertia about the origin. Note that $x \cdot N(x)>0$ because the domain is starlike. In higher dimensions we will later define $G_0$ and $G_1$ differently, although the definitions will reduce to \eqref{Gin2dim} in $2$ dimensions. 

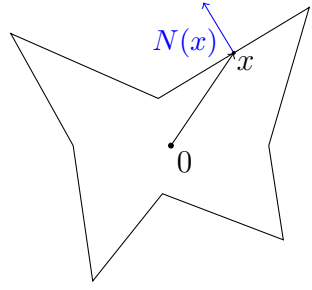
\begin{figure}
\begin{center}
\begin{tikzpicture}[scale=1.3]
  \draw (0,0) +(0:1) coordinate (g) -- +(45:2) coordinate (a) -- +(105:0.5) coordinate (b) -- +(145:2) coordinate (e) -- +(180:1) coordinate (f) -- +(240:1.6) coordinate (c) -- +(260:0.5) coordinate (d) -- +(320:1.5) -- +(0:1);

  \draw[black,->] (0,0) -- ($(a)!0.5!(b)$) coordinate (x) node [below right=-3pt] {$x$};
  \draw[blue,->] (x) -- ($(x)!0.6cm!90:(a)$) node[pos=0.5,below left=-4pt] {\small $N(x)$};
  \fill (x) circle (0.02);




  \fill (0,0) circle (0.03) node [below right=-2pt] {$0$};
\end{tikzpicture}
\end{center}
\caption{A starlike domain with outer normal $N(x)$.}
\label{fig:starlike}
\end{figure}

Let $G = \max \{ G_0,G_1 \}$. Then $G \geq 1$ for all starlike domains with equality for centered disks, by \autoref{le:leq1} below. Thus one may regard the value of $G$ as measuring the deviation of the domain from roundness. This deviation can arise in two ways: a highly oscillatory starlike boundary would make $G_0$ large, while an elongated boundary (such as an eccentric ellipse) would make $G_1$ large.

Now we can state the main result in the plane. We show that the disk maximizes eigenvalues of the Laplacian under suitable geometric scaling. 
\begin{theorem}[Dirichlet in $2$ dimensions] \label{Deigen2}
Suppose the function $R(\theta)$ is $2\pi$-periodic, positive, and Lipschitz continuous, and consider the starlike  domain $\Omega = \{ re^{i\theta} : 0 \leq r < R(\theta) \}$. Let $n \geq 1$. 

Then each of the following scale invariant eigenvalue functionals achieves its maximum value when the domain $\Omega$ is a centered disk:
\[
\lambda_1 A/G_0, \qquad \lambda_2 A/G_0, \qquad
(\lambda_1^s + \cdots + \lambda_n^s)^{1/s} \, A/G, \qquad
\sqrt[n]{\lambda_1 \lambda_2 \cdots \lambda_n} \, A/G,
\]
for each exponent $0 < s \leq 1$. Further, if $\Phi : \R_+ \to \R$ is concave and increasing then $\sum_{j=1}^n \Phi(\lambda_j A/G)$ is maximal when $\Omega$ is a disk centered at the origin

Hence the partial sums of the spectral zeta function and trace of the heat kernel are minimal when $\Omega$ is a centered disk. That is, the functionals
\[
\sum_{j=1}^n (\lambda_j A/G)^s \qquad \text{and} \qquad
\sum_{j=1}^n \exp(-\lambda_j At/G)
\]
attain their smallest value when $\Omega$ is a centered disk, for each $s<0<t$.
\end{theorem}
The theorem is better for the first and second eigenvalues than for other functionals, in the sense that we normalize $\lambda_1 A/G_0$ and $\lambda_2 A/G_0$ with the quantity $G_0$ instead of with $G$, where obviously $G_0 \leq G$ by definition. 

It is natural in the theorem to multiply $\lambda_j$ by $A$, because $\lambda_j$ scales like $1/A$. (Intuitively, low frequencies come from large drums.)

Note the result for $\lambda_2$ follows immediately from the one for $\lambda_1$, because the ratio $\lambda_2/\lambda_1$ is maximal on the disk by Ashbaugh and Benguria's ``sharp PPW inequality'' \cite{AB92}.   

The theorem improves on the standard ``inradius bounds'' for $\lambda_1$ and $\lambda_2$, on convex domains, as we now show. Write $\D_{in}$ for the largest open disk centered at the origin and contained in $\Omega$. Then $\lambda_j(\Omega) \leq \lambda_j(\D_{in})$ for all $j$, by domain monotonicity of the Dirichlet spectrum. \autoref{Deigen2} implies this inradius bound for $j=1,2$, as one checks by using that $A_{in} \leq A/G_0$ (\autoref{le:alt}); here $A_{in}$ is the area of $\D_{in}$.

Our theorem significantly extends the only known result of its type, which is the case ($n=1$) of the fundamental tone $\lambda_1$ with Dirichlet boundary conditions. That case is due to P\'{o}lya and Szeg\H{o} in $2$ dimensions and Freitas and Krej{\v{c}}i{\v{r}}{\'{\i}}k in higher dimensions, as explained after \autoref{Deigen}. 

To treat higher eigenvalues, we need a fundamentally new idea: we need to transform $\Omega$ into a disk while controlling  \emph{angular} information in the Rayleigh quotients of the eigenfunctions. Any such transformation will change the Rayleigh quotients substantially, and so we must devise a scheme for extracting the geometric effect and leaving behind the portion of the Rayleigh quotient that corresponds to the eigenfunction of the disk.

We construct a geometric transformation that maps linearly on rays and has constant Jacobian. As \autoref{fig:transplant} indicates, wherever the transformation stretches radially it must compress angularly. \autoref{sec:setup} gives the precise definition. We will extract the geometric contribution to the Rayleigh quotient of the trial function on $\Omega$ by composing the transplanted eigenfunction with an arbitrary orthogonal transformation $U$ of the ball and then averaging over all such $U$ (see \autoref{pr:Q23} and \autoref{Deigen_proof}). The constant Jacobian requirement is used here to guarantee that transplanting orthogonal eigenfunctions from the ball will yield orthogonal trial functions on $\Omega$.

\begin{figure}
\begin{center}
\begin{tikzpicture}[scale=1.5]
  \begin{scope}[xshift=5cm]
        \draw[red,->] (145:1.2) node[above=24pt] {rotation $U$} arc (145:90:1.2);
    \draw (0,0) circle (0.95);
    \fill (1.05,0.3) circle (0) node [right] {\small $u(x)$};
        \fill (1.05,-0.1) circle (0) node [right] {\small eigenfunction};
    \draw[loosely dashed] (0,0) circle (0.95*2/3);
    \draw[dashed] (0,0) circle (0.95/3);
    \draw [green!50!black] (0,0) -- (0.95,0);
    \begin{scope}
      \draw [green!50!black](0,0) -- (30:0.95) node [right=-2pt,pos=1,scale=1.2] {\tiny $1$};
      \clip (0,0) -- (30:0.95) -- (0.95,0) --cycle;
      \draw [green!50!black] (0,0) circle (0.57);
      \draw[green!50!black] (0.42,0.1) node[scale=1.2] {\tiny $\phi$};
    \end{scope}
    \fill (0,0) circle (0.03) node [below,scale=0.8] {\tiny $0$};
  \end{scope}
    \draw[->,blue] (1.3,0.4) .. controls (2.1,0.7) and (3,0.7) .. (3.8,0.4) node [above,pos=0.5] {$T$} node[below=5pt,pos=0.5] {\small linear on each ray,} node[below=14pt,pos=0.5] {\small area preserving};
  \begin{scope}[xscale=-1,xshift=-5.5cm]
    \draw[smooth] (4.5,0) .. controls (4.5,1) .. (5,1) .. controls (6,1) .. (6,0) .. controls (6,-1) .. (5,-1) .. controls (4.5,-1) .. (4.5,0);
    \draw[smooth,loosely dashed,xshift=-5cm,scale=2/3,xshift=10cm] (4.5,0) .. controls (4.5,1) .. (5,1) .. controls (6,1) .. (6,0) .. controls (6,-1) .. (5,-1) .. controls (4.5,-1) .. (4.5,0);
    \draw[smooth,dashed,xshift=-5cm,scale=1/3,xshift=25cm] (4.5,0) .. controls (4.5,1) .. (5,1) .. controls (6,1) .. (6,0) .. controls (6,-1) .. (5,-1) .. controls (4.5,-1) .. (4.5,0);
  \begin{scope}[xscale=-1,xshift=-10cm]
    \draw[green!50!black] (5,0) -- (5.5,0);
    \fill (5,0) circle (0.03) node [below] {\tiny $0$};
        \fill (2.4,0.3) circle (0) node [right] {\small $u \big( UT(x) \big)$};
         \fill (2.1,-0.1) circle (0) node [right] {\small trial function};
   \draw[green!50!black] (5,0) -- ++(60:0.95) node [right=-2pt,pos=1,scale=1.2] {\tiny $R(\theta)$};
    \clip (5,0) -- +(1,0) -- +(60:1) -- cycle;
    \draw[green!50!black] (5,0) circle (0.4);
    \draw[green!50!black] (5.22,0.1) node [scale=1.2] {\tiny $\theta$};
  \end{scope}
  \end{scope}
\end{tikzpicture}
\end{center}
   \vspace{-0.3cm}
\caption{A linear-on-rays transformation from a domain $\Omega$ of area $\pi$ to the unit disk. To insure that the mapping preserves area locally, we require $R(\theta)^2 \, d\theta = d\phi$.}
\label{fig:transplant}
\end{figure}
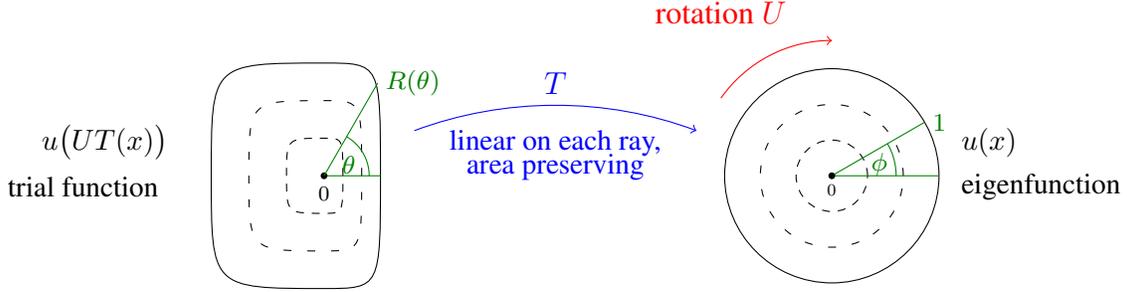

Note that P\'{o}lya and Szeg\H{o}'s result on the first eigenvalue was proved by a linear-on-rays transformation that does \emph{not} distort angles and hence does not preserve area. In other words, they took $\theta=\phi$ in \autoref{fig:transplant}. They also did not average over rotations. Instead they relied on the special fact that the first Dirichlet eigenfunction of the disk is radial.

\subsection*{Perturbations of the disk}
To make the theorem more concrete, we examine the case of nearly circular domains.
Suppose $P(\theta)$ is a Lipschitz continuous, $2\pi$-periodic function. Define a plane domain $\Omega_\e = \{ re^{i\theta} : 0 \leq r < 1+\e P(\theta) \}$, and assume $\e$ is small enough that the radius $1+\e P(\theta)$ is positive for all $\theta$. We may regard $\Omega_\e$ as a perturbation of the unit disk $\D$, since $\Omega_0=\D$. Let $j_{0,1} \simeq 2.4$ be the first zero of the Bessel function $J_0$, and recall that $\lambda_1(\D)=j_{0,1}^2$.
\begin{corollary}[Nearly circular domains] \label{perturb2}
The first eigenvalue of the domain $\Omega_\e$ is bounded above and below in terms of the boundary perturbation $P$:
\begin{align}
1 \leq \frac{\lambda_1(\Omega_\e) A(\Omega_\e)}{j_{0,1}^2 \pi}
& \leq 1 + \e^2 \int_0^{2\pi} \frac{P^\prime(\theta)^2}{\big( 1+\e P(\theta) \big)^2} \, \frac{d\theta}{2\pi} \label{eq:pertest} \\
& = 1 + \e^2 \left( \int_0^{2\pi} P^\prime(\theta)^2 \, \frac{d\theta}{2\pi} \right) + O(\e^3) \notag
\end{align}
as $\e \to 0$ with $P$ fixed. 

Eigenvalue sums satisfy a similar upper bound, for $n \geq 1$ and $s \in (0,1]$:
\begin{align*}
& \frac{\big(\sum_{j=1}^n \lambda_j(\Omega_\e)^s \big)^{\! 1/s} A(\Omega_\e)}{\big( \sum_{j=1}^n \lambda_j(\D)^s \big)^{\! 1/s} A(\D)} \\
& \leq \max \left\{ 1 + \int_0^{2\pi} \frac{\e^2 P^\prime(\theta)^2}{\big( 1+\e P(\theta) \big)^2} \, \frac{d\theta}{2\pi} \, , \frac{\int_0^{2\pi} \big( 1 + \e P(\theta)\big)^{\! 4} \, d\theta/2\pi}{\big[ \int_0^{2\pi} \big( 1 + \e P(\theta)\big)^{\! 2} \, d\theta/2\pi \big]^{\! 2}}\right\} \\
& = 1 + O(\e^2) .
\end{align*}
\end{corollary}
The lower bound on the first eigenvalue in \eqref{eq:pertest} is the famous Faber--Krahn inequality. The upper bounds are immediate from \autoref{Deigen2}, by taking $R=1+\e P$ in the formula for $G_0$ in \autoref{le:G2dim} and remembering that our two definitions of $G_0$ agree in $2$ dimensions (\autoref{le:alt}).

The upper bound on the first eigenvalue in \eqref{eq:pertest} is equivalent to the estimate of P\'{o}lya and Szeg\H{o} \cite[pp.~14--15, 91--92]{PS51}. The upper bound on eigenvalue sums is new.

Let us compare the upper bound on the first eigenvalue with recent work of van den Berg \cite[Theorem~1(ii)]{vdB12}. He obtained an estimate of the form 
\begin{equation} \label{eq:vdB}
1 + C_2 \big( \lVert P \rVert_2 \lVert P^\prime \rVert_2 + \lVert P \rVert_2^2 \big) \e^2 + C_3 \lVert P^\prime \rVert_2^2 \, \e^3 ,
\end{equation}
for perturbations normalized by $\lVert P \rVert_\infty = 1$. Note his formula involves $\lVert P \rVert_2 \lVert P^\prime \rVert_2$ at second order whereas our corollary has $\lVert P^\prime \rVert_2^2$.  When applied to the main example in van den Berg's paper, his estimate beats our estimate \eqref{eq:pertest} by a factor of $\e^{1/2}$ because his $P$ depends on $\e$ with $\lVert P \rVert_2 = O(\e^{1/4})$ and $\lVert P^\prime \rVert_2 = O(\e^{-1/4})$. On the other hand, when applied to the ``uniform'' perturbation $P \equiv -1$, equality holds in our estimate \eqref{eq:pertest} whereas \eqref{eq:vdB} is not exact, due to the contribution of $\lVert P \rVert_2^2$.

Next let us contrast with Rayleigh's second order perturbation expansion of $\lambda_1$. See either his formal derivation \cite[{\S}210]{S94}, or P\'{o}lya and Szeg\H{o}'s account \cite[pp.~132--133]{PS51}, or Henry's more modern approach in all dimensions \cite[p.~35]{H05}. This expansion gives no error bounds, and it holds only for $P$ fixed with $\e$ tending to $0$. Both our corollary and van den Berg's work establish approximation bounds, and allow $P$ to vary with $\e$. Rayleigh's perturbation formula is better in one respect, though, because it gives an exact second order term ($\e^2$-term) for $\lambda_1 A$. This term behaves like $\sum |n| |\widehat{P}(n)|^2$, that is, like the square of the $H^{1/2}$-norm of the boundary perturbation.  In contrast, both our corollary and van den Berg's work have second order terms that are bigger, being controlled by the square of the $H^1$-norm, $\sum n^2 |\widehat{P}(n)|^2$.

\subsection*{Prior work, and the new methods} The idea of obtaining eigenvalue bounds by transforming a domain and averaging over rotations appeared already in Laugesen and Morpurgo's  $2$-dimensional conformal mapping approach \cite{LM98}. Their averaging task was much easier, though, since it needed only subharmonicity of the modulus of an analytic function. Further, their ``reciprocal eigenvalue'' results are inherently less powerful than the methods of this paper since, for example, they cannot yield the heat trace for all $t>0$.

More recently, sharp eigenvalue bounds on linear images of rotationally symmetric domains (such as regular polygons) were obtained by the authors and collaborators \cite{L12,LLR12,LS11b,LS11c}. For example, they showed that the centered equilateral triangle maximizes $(\lambda_1 + \cdots + \lambda_n)A/G_1$ among all triangles. The averaging in those papers takes place over discrete groups of rotations (such as $3$-fold rotations for triangles), and relies on ``tight frame'' identities which are special cases of Schur's Lemma from representation theory. The transformations in those papers are globally linear, and so are simpler than the linear-on-rays transformations constructed in this paper. This simplicity comes at the cost of a severely restricted class of image domains.

We must push beyond these existing averaging methods, in this paper. One serious obstacle is the nonlinear nature of our transformation, which causes the rotation matrix $U$ to appear multiple times in the transformed Rayleigh quotient, both inside and outside the derivative of the transformation. We describe how to overcome these obstacles in \autoref{Deigen_proof}.

\subsection*{Faber--Krahn and Luttinger bounds in the reverse direction} Rayleigh conjectured in 1877, and Faber and Krahn proved in the the 1920s, that
\[
\text{$\lambda_1 A$ is minimal for the disk.}
\]
Many proofs and extensions are known \cite{B80,BL08,K06,PS51}. The result holds trivially for the first Neumann eigenvalue $\mu_1$, which equals zero for each domain. The result holds also for the first Robin eigenvalue, assuming a positive Robin parameter, by work of Bossel and Daners \cite{B86,D06} that was improved to irregular domains by Bucur and Giacomini \cite{BG10}.

This Rayleigh--Faber--Krahn inequality does not extend to sums and products of eigenvalues, as the results in this paper do. It does extend to the spectral zeta function and heat trace, for $s<-1$ and $t>0$ respectively, each of them being maximal for the disk of the same area under Dirichlet boundary conditions. This important extension is due to Luttinger \cite{L73}, whose multiple integral rearrangement techniques proved remarkably fertile in joint work with Brascamp and Lieb \cite{BLL74}. Note that letting $t \to \infty$ in the heat trace inequality yields a new proof of the Faber--Krahn theorem. 

Luttinger's methods and results do \emph{not} extend to Neumann boundary conditions. In fact, for large $t$ the area-normalized Neumann heat trace equals approximately $1+e^{-\mu_2 At}$, which is minimal for the disk (not maximal) by Szeg\H{o}--Weinberger's result on the second Neumann eigenvalue \cite{W56}. One naturally conjectures that the Neumann heat trace is minimal for the disk of the same area, for each $t>0$. This problem remains open. For the analogous problem of the heat trace on the sphere, Morpurgo \cite{M96} has proved local minimality at the round metric. 

To put this paper in context, then, one may regard our results as being analogous to the classical Faber--Krahn and Luttinger results except with the direction of their inequalities reversed. Such reversal is made possible by introducing the geometric factor $G$ into the geometric scaling. Further, our theorems improve in three respects on the Faber--Krahn and Luttinger inequalities, because they hold: for finite sums and products of eigenvalues, for each \emph{partial sum} of the spectral zeta function and heat trace, and for Neumann and Robin boundary conditions in addition to Dirichlet.

\section{\bf The volume preserving transformation, and geometric factors}
\label{sec:setup}

Write $\B=\B^d$ for the unit ball centered at the origin, $\Sp=\Sp^{d-1}$ for the unit sphere, and $\R_+=(0,\infty)$ for the positive half-axis. 

We say a domain $\Omega$ in $\Rd$ is \emph{Lipschitz-starlike} if it can be expressed in the form
\[
\Omega = \{ r\xi : \xi \in \Sp, 0 \leq r < R(\xi) \}
\]
for some positive, Lipschitz continuous function $R(\xi)$ on $\Sp$. Call $R$ the \emph{radius function} of $\Omega$. The \emph{gauge function} is its reciprocal,
\[
\Gamma = \frac{1}{R} .
\]
Write $V$ for volume in $\Rd$. Notice $V(\Omega) = \frac{1}{d} \int_\Sp R(\xi)^d \, dS(\xi)$. 

\subsection*{The volume preserving (constant Jacobian) transformation} Our work relies on a map from the Lipschitz-starlike domain $\Omega$ to the ball that preserves volume locally (up to a scale factor), and is linear on each ray from the origin. See \autoref{fig:transplant} for an example in $2$ dimensions. First we construct a ``boundary homeomorphism'' $H$ associated with the radius function $R$ of $\Omega$. 
\begin{lemma} \label{homexist}
There exists a bi-Lipschitz homeomorphism $H : \Sp \to \Sp$ that distorts surface area in proportion to the $d$-th power of the radius function:
\begin{equation} \label{eq:distort}
\text{Jac}_H(\xi) = \frac{V(\B)}{V(\Omega)} R(\xi)^d .
\end{equation}
\end{lemma}
We prove the lemma in \autoref{sec:Hexist}. Simply note at this stage that the left and right sides of \eqref{eq:distort} both integrate over $\Sp$ to yield $|\Sp|$. 

Now define the mapping $T : \Omega \to \B$ by mapping linearly in each direction and transforming directions with $H$; that is, define 
\begin{equation} \label{eq:Tdef}
T(r\xi) = \frac{r}{R(\xi)} H(\xi)
\end{equation}
for vectors $\xi \in \Sp$ and numbers $r \in [0,R(\xi))$. One can check that $T$ is a bi-Lipschitz homeomorphism of $\Omega$ to $\B$. Its Jacobian determinant is constant, with
\[
\text{Jac}(T) \equiv V(\B)/V(\Omega) , \qquad
\text{Jac}(T^{-1}) \equiv V(\Omega)/V(\B) ,
\]
as one deduces from the definition \eqref{eq:Tdef} and Jacobian formula \eqref{eq:distort}. 

The constant Jacobian property of $T$ will be essential later, when we transplant an orthonormal collection of eigenfunctions on the ball to a collection of functions on $\Omega$. The transplanted functions will remain  orthogonal, thanks to the constant Jacobian condition, and so we can use them as trial functions in the Rayleigh principle for the eigenvalue sum. 

%
%

\subsection*{The geometric factors} From now on, we extend $R$ and $H$ by homogeneity to be defined not just on the unit sphere but on all nonzero vectors:
\[
R(r\xi) = R(\xi), \qquad H(r\xi) = H(\xi),
\]
for all $r>0$. Thus it makes sense to speak of the gradient vector $\nabla R$, and the derivative matrix $DH$.

Given a real matrix $M$, write its Hilbert--Schmidt norm as
\[
\lVert M \rVert_{HS}=\big( \sum_{j,k} M_{jk}^2 \big)^{\! 1/2} = (\tr M^\dagger M)^{1/2} ,
\]
where $M^\dagger$ denotes the transposed matrix. All matrices in this paper will be real.

Now define two geometric quantities
\begin{align}
G_0(\Omega) & =  \frac{\frac{1}{|\Sp|} \int_\Sp \big[ R(\xi)^{d-2} + \big|\nabla R(\xi) \big|^2 R(\xi)^{d-4} \big] \, dS(\xi)}{\Big( \frac{1}{|\Sp|} \int_\Sp R(\xi)^d \, dS(\xi) \Big)^{\! (d-2)/d} } , \label{eq:G0def} \\
G_1(\Omega) & =   \frac{\frac{1}{|\Sp|} \int_\Sp \frac{\lVert DH(\xi) \rVert_{HS}^2}{d-1} R(\xi)^{d-2} \, dS(\xi)}{\Big( \frac{1}{|\Sp|} \int_\Sp R(\xi)^d \, dS(\xi) \Big)^{\! (d-2)/d} } . \label{eq:G1def}
\end{align}
Clearly $G_0$ and $G_1$ are scale invariant, meaning that $G_i(\Omega)=G_i(a\Omega)$ for all $a>0$, since $a\Omega$ has radius function $aR$.

Alternative formulas for $G_0$ and $G_1$ of a geometric nature will be developed in \autoref{sec:understanding}. There we express $G_0$ in terms of a support-type integral over the boundary that was employed previously by P\'{o}lya and Szeg\H{o}, and we show in two dimensions that $G_1=2\pi I_\text{origin}/A^2$. Thus these alternative formulas recover the definitions \eqref{Gin2dim} that we used in the plane, and show that in $2$ dimensions, both $G_0$ and $G_1$ depend only on the \emph{shape} of $\Omega$ and on the choice of origin. 

In higher dimensions, $G_1$ depends also on the choice of homeomorphism $H$.

\smallskip
\emph{Example.} If $\Omega$ is a centered ball one has $R \equiv \text{const.}$, so that $G_0=1$. By convention, for a centered ball we choose the homeomorphism $H$ to be the identity on the sphere, so that  $G_1=1$.
\begin{lemma} \label{le:leq1}
The geometric quantities are always at least $1$ in value:
\[
G_0 \geq 1 \qquad \text{and} \qquad G_1 \geq 1 .
\]
Equality statement: $G_0=1$ if and only if $\Omega$ is a centered ball, and $G_1=1$ if and only if $\Omega$ is a centered ball and $H$ is an orthogonal transformation of the sphere.
\end{lemma}
This lemma helps us interpret the main theorem below. We do not otherwise need the lemma though, and so we defer its proof to \autoref{sec:understanding}.

\section{\bf Main results}
\label{sec:main}

First we extend the eigenvalues estimates in \autoref{Deigen2} to all dimensions. Let
\[
G = \max \{ G_0 , G_1 \} .
\]
\begin{theorem}[Dirichlet] \label{Deigen}
Assume $\Omega$ is a Lipschitz-starlike domain in $\Rd, d \geq 2$, and let $n \geq 1$. Then the scale invariant eigenvalue functionals
\[
\lambda_1 V^{2/d}/G_0, \quad \lambda_2 V^{2/d}/G_0, \quad
(\lambda_1^s + \cdots + \lambda_n^s)^{1/s} \, V^{2/d}/G, \quad
\sqrt[n]{\lambda_1 \lambda_2 \cdots \lambda_n} \, V^{2/d}/G,
\]
are maximal when $\Omega$ is a centered ball, for each exponent $0 < s \leq 1$. Further, if $\Phi : \R_+ \to \R$ is concave and increasing then $\sum_{j=1}^n \Phi(\lambda_j V^{2/d}/G)$ is maximal when $\Omega$ is a centered ball. Hence for $s<0< t$ the functionals
\[
\sum_{j=1}^n (\lambda_j V^{2/d} / G)^s \qquad \text{and}  \qquad
\sum_{j=1}^n \exp(-\lambda_j V^{2/d} \, t/G) 
\]
are minimal when $\Omega$ is a centered ball.  

Equality statement for the first eigenvalue: if $\lambda_1 V^{2/d} /G_0 \big|_\Omega = \lambda_1 V^{2/d} /G_0 \big|_\B$ and  $R$ is $C^2$-smooth then $\Omega$ is a centered ball.
\end{theorem}
The proof is in \autoref{Deigen_proof}.

The only part of the theorem known previously was the estimate on the first eigenvalue. This extremal result for $\lambda_1 V^{2/d} /G_0$ was proved by P\'{o}lya and Szeg\H{o} \cite[pp.~14--15, 91--92]{PS51} in $2$ dimensions, and by Freitas and Krej{\v{c}}i{\v{r}}{\'{\i}}k \cite[Theorem~3]{FK08} in higher dimensions. The geometric factors in those papers look different from our $G_0$, but they are equivalent, as we explain in \autoref{sec:understanding}.

Note that for the first (and second) eigenvalue, the conclusion of our theorem is stronger than for the general case, because it uses $G_0$ instead of $G$. The underlying reason is that the first eigenfunction of a ball is purely radial, so that our proof does not depend on the angular information encoded in the homeomorphism $H$ and factor $G_1$.

We strengthen the theorem in \autoref{sec:improvements} by adapting the geometric factor to each eigenvalue. There we replace $G$ with a convex combination of $G_0$ and $G_1$ (rather than their maximum).

\subsection*{Perturbations of the ball}
Let us see what \autoref{Deigen} says for nearly spherical domains. Suppose $P(\xi)$ is a Lipschitz continuous function on the sphere $\Sp$, and define a domain $\Omega_\e = \{ r\xi : 0 \leq r < 1+\e P(\xi) \}$, assuming $0<1+\e P(\xi)$ for all $\xi$. 
\begin{corollary}[Nearly spherical domains] \label{perturbhigher}
The first eigenvalue of the domain $\Omega_\e$ can be bounded above and below in terms of the radial perturbation $P$:
\begin{align*}
1 \leq \frac{\lambda_1(\Omega_\e) V(\Omega_\e)^{2/d}}{\lambda_1(\B) V(\B)^{2/d}} & \leq G_0(\Omega_\e) \\
& = 1 + \Big( \int_\Sp |\nabla P|^2 \, \frac{dS}{|\Sp|} - (d-2) \int_\Sp (P-\overline{P})^2 \, \frac{dS}{|\Sp|} \Big) \e^2 + O(\e^3)
\end{align*}
as $\e \to 0$ with $P$ fixed, where $\overline{P}=\int_\Sp P \, dS/|\Sp|$ is the mean value of the perturbation.
\end{corollary}
The lower bound is simply the Faber--Krahn result. The upper bound appears not to have been stated before. It follows by straightforward calculations from \autoref{Deigen}, simply substituting $R=1+\e P$ into the definition \eqref{eq:G0def} of $G_0$. It can be compared with van den Berg's result \cite[Theorem~1(ii)]{vdB12}, just like in $2$ dimensions --- see the remarks after \autoref{perturb2}.

Amusingly, the corollary implies a Poincar\'{e} inequality on the sphere, since the $\e^2$-term is necessarily nonnegative. 

\subsection*{Neumann and Robin boundary conditions}
Denote the Neumann eigenvalues by $\mu_j$, assuming that $\partial \Omega$ is Lipschitz so that the spectrum exists and is discrete. Write $u_j$ for corresponding orthonormal eigenfunctions. Then
\[
\begin{cases}
-\Delta u_j = \mu_j u_j \;\; \text{in $\Omega$} \\
\hfill \frac{\partial u_j}{\partial n}  = 0 \;\; \text{on $\partial \Omega$}
\end{cases}
\]
and
\[
0 = \mu_1 < \mu_2 \leq \mu_3 \leq \dots .
\]
We will ignore the first eigenvalue, in the next theorem, since $\mu_1=0$ for every domain.
\begin{theorem}[Neumann] \label{Neigen}
Assume $\Omega$ is a Lipschitz-starlike domain in $\Rd, d \geq 2$. Suppose $\Phi : \R_+ \to \R$ is concave and increasing, and let $n \geq 2$. Then the scale invariant eigenvalue functional $\sum_{j=2}^n \Phi(\mu_j V^{2/d} /G)$ is maximal when $\Omega$ is a centered ball. 

In particular, for $0 < s \leq 1$ the functionals
\[
\mu_2 V^{2/d} /G, \qquad
(\mu_2^s + \cdots + \mu_n^s)^{1/s} \, V^{2/d} /G, \qquad
\sqrt[n-1]{\mu_2 \cdots \mu_n} \, V^{2/d} /G,
\]
are maximal when $\Omega$ is a centered ball. For $s<0<t$ the functionals
\[
\sum_{j=2}^n (\mu_j V^{2/d} /G)^s \qquad \text{and} \qquad 
\sum_{j=2}^n \exp(-\mu_j V^{2/d} \, t/G) 
\]
are minimal when $\Omega$ is a centered ball.

Equality statement for the first nonzero eigenvalue: if $R$ and $H$ are $C^2$-smooth and $\mu_2 V^{2/d} /G \big|_\Omega = \mu_2 V^{2/d} /G \big|_\B$ then $\Omega$ is a centered ball.
\end{theorem}
\autoref{Neigen_proof} has the proof. Because $G \geq 1$, the bound on $\mu_2 V^{2/d} /G$ in \autoref{Neigen} follows from the Szeg\H{o}--Weinberger theorem that $\mu_2 V^{2/d}$ is maximal for the ball (see \cite{W56}, or  \cite[Theorem 7.1.1]{He06}). Note that \autoref{Neigen} holds for higher eigenvalue functionals too, which the Szeg\H{o}--Weinberger theorem does not.

\autoref{Neigen} might hold with $G$ replaced by the smaller quantity $G_1$ (the moment of inertia type functional), as we have conjectured elsewhere \cite[\S4]{LS11b}.  

Next we turn to Robin boundary conditions. The Robin eigenvalue problem is
\[
\begin{cases}
-\hbar^2 \Delta u_j = \rho_j u_j \;\; \text{in $\Omega$,} \\
\hfill \hbar^2 \frac{\partial u_j}{\partial n} + \sigma u_j = 0 \;\; \text{on $\partial \Omega$,}
\end{cases}
\]
with eigenvalues
\[
\rho_1 < \rho_2 \leq \rho_3 \leq \dots ,
\]
where $\sigma \in L^\infty(\partial \Omega)$ is the \emph{Robin parameter} and $\hbar>0$ is the \emph{Planck constant}. Existence and discreteness of the Robin spectrum under these assumptions
follows from the usual quadratic form approach; see for example \cite[Chapter~5]{L12a}. Note the Robin parameter $\sigma$ will be fixed in our work. For interesting asymptotic behavior of Robin eigenvalues as $\sigma$ varies and approaches $\pm \infty$, see \cite{DK10,Ko12} and references therein. 

The Robin eigenvalues reduce to Neumann when $\hbar=1, \sigma \equiv 0$.

In the Dirichlet and Neumann eigenvalue problems we took $\hbar=1$. That causes no loss of generality, since one can always adjust the value of $\hbar$ by rescaling the domain. In our Robin result below, though, the Planck constant and Robin parameter will be multiplied by different geometric factors. Accordingly we write $\rho_j = \rho_j(\Omega,\hbar,\sigma)$ to display the dependence of the $j$th Robin eigenvalue on the domain, Planck constant and Robin parameter.

Our theorem will involve a new geometric factor,
\begin{equation} \label{eq:GRobindef}
G_\Robin =  \left( \frac{|\partial \Omega|/V(\Omega)^{(d-1)/d}}{|\partial \B|/V(\B)^{(d-1)/d}} \right)^{\! \! 2} .
\end{equation}
Clearly $G_\Robin \geq 1$ by the isoperimetric inequality, with equality if and only if $\Omega$ is a ball. And $G_0$ is larger than $G_\Robin$:
\begin{lemma} \label{le:leq1Robin}
$G_0 \geq G_\Robin \geq 1$.
\end{lemma}
The lemma was proved in $2$ dimensions by Aissen \cite[Theorem~1]{A58}. Our proof appears in \autoref{sec:understanding}, and is valid in all dimensions.

Now we can state our sharp upper bound on Robin eigenvalues.
\begin{theorem}[Robin] \label{Reigen}
Assume $\Omega$ is a Lipschitz-starlike domain in $\Rd, d \geq 2$. Suppose $\Phi : \R \to \R$ is concave and increasing, and let $n \geq 1$. Then
\[
\sum_{j=1}^n \Phi \big( \rho_j(\Omega,\hbar V^{1/d} /G^{1/2},\sigma V^{1/d} /G_\Robin^{1/2}) \big)
\]
is maximal when $\Omega$ is a centered ball and $\sigma$ is replaced by its average value.

For the first eigenvalue one has a stronger result (with $G_0$ instead of $G$):
\begin{equation} \label{eq:Robinfirst}
\rho_1 \big( \Omega,\hbar \overline{R}/G_0^{1/2},\sigma \overline{R}/G_\Robin^{1/2} \big) \\
\leq \rho_1( \B,\hbar,\overline{\sigma})
\end{equation}
where $\overline{R}$ is the radius of a ball having the same volume as $\Omega$ and $\overline{\sigma}=\int_{\partial \Omega} \sigma \, dS/|\partial \Omega|$ is the average value of the Robin parameter. If equality holds in \eqref{eq:Robinfirst} and if $R$ is $C^2$-smooth, then $\Omega$ is a centered ball.
\end{theorem}
See \autoref{Reigen_proof} for the proof. 

Note that if $\sigma>0$ then the Robin eigenvalues are all positive, in which case $\Phi$ need only be concave and increasing on the half-axis $\R_+$. 
   
A particularly simple corollary holds for the ball: averaging the Robin parameter increases the eigenvalue functionals on a ball, with 
\[
\sum_{j=1}^n \Phi \big( \rho_j(\B,\hbar,\sigma) \big) \leq \sum_{j=1}^n \Phi \big( \rho_j(\B,\hbar,\overline{\sigma}) \big) .
\]

\medskip
\noindent \emph{Remark.} Bareket \cite[Appendix~A]{B77} applied the P\'{o}lya--Szeg\H{o} trial function technique with a constant parameter $\sigma<0$ and got an upper bound on $\rho_1$. She did not attach geometric factors to the Planck constant or Robin parameter, though, and so her bound looks different from ours in \eqref{eq:Robinfirst}.

Bareket raised an analogue of Rayleigh's Conjecture for negative Robin parameter, namely that 
\[
\rho_1(\Omega,1,\overline{\sigma}) \leq \rho_1(\B,1,\overline{\sigma}) ,
\]
whenever $\overline{\sigma}<0$ is constant and $\Omega$ has the same volume as $\B$. Notice the ball is the maximizer here, which is the opposite of the Faber--Krahn type result by Bossel and Daners that holds when $\overline{\sigma} \geq 0$. Let us compare Bareket's conjecture with our result \eqref{eq:Robinfirst} for the first eigenvalue, which says (when $\hbar=1, \overline{R}=1$ and $\sigma \equiv \overline{\sigma} < 0$) that 
\[
\rho_1 \big( \Omega,1/G_0^{1/2},\sigma/G_\Robin^{1/2} \big) \leq \rho_1( \B,1,\overline{\sigma}) .
\]
The factors $G_0$ and $G_\Robin$ are both greater than $1$, and so the Planck constant and Robin parameter are smaller in magnitude on the left side of our inequality than on the left side of Bareket's conjecture. In particular, our Robin parameter is less negative than Bareket's. Thus while our Planck constant tends to make the Rayleigh quotient smaller than it would be with Bareket's Planck constant, our Robin parameter tends to make it bigger. Hence our result is most likely not comparable to her conjecture.

\section{\bf Averaging over rotations and reflections}

Our proofs will involve averaging over the group of all orthogonal transformations.

Write $\gamma=\gamma_d$ for the Haar probability measure on the group $O(d)$ of orthogonal, real, $d \times d$ matrices. Let ``$\Id$'' denote the identity matrix.
\begin{lemma}[Averaging in a conjugacy class] \label{le:averaging}
Supose $M$ is a $d \times d$ real symmetric matrix, for some $d \geq 1$. Then
\begin{equation} \label{eq:traceav}
  \int_{O(d)} U^{-1} M U \, d\gamma(U) = \frac{1}{d} \tr(M) \Id .
\end{equation}
Hence for each column vector $m \in \Rd$,
\[
  \int_{O(d)} \begin{pmatrix} U & 0 \\ 0 & 1 \end{pmatrix}^{\! \! -1} \begin{pmatrix} M & m \\ m^\dagger & 0 \end{pmatrix} \begin{pmatrix} U & 0 \\ 0 & 1 \end{pmatrix} \, d\gamma(U) = \frac{1}{d} \tr(M) \begin{pmatrix} \Id & 0 \\ 0 & 0 \end{pmatrix} .
\]
\end{lemma}
The lemma is a special case of Schur's Lemma from representation theory. We give a short proof, for the sake of completeness.
\begin{proof}[Proof of \autoref{le:averaging}]
Denote the left side of formula \eqref{eq:traceav} by
\[
L= \int_{O(d)} U^{-1} M U \, d\gamma(U) .
\]
For any $U \in O(d)$, we have $LU=UL$ by invariance of Haar measure. Let $\alpha$ be a real eigenvalue of $L$ with eigenvector $w \in \Rd$ (using here that $L$ is symmetric, by symmetry of $M$ and orthogonality of $U$). Then
\[
  L(Uw)=U(Lw)=U(\alpha w)=\alpha(Uw).
\]
Hence each vector in the orbit $\{ Uw : U \in O(d) \}$ is an eigenvector of $L$ with eigenvalue $\alpha$. The orbit spans all of $\Rd$, and so $L$ equals $\alpha$ times the identity. Taking the trace yields
\[
\alpha d = \tr L=\int_{O(d)} \tr(U^{-1} M U) \, d\gamma(U) =\int_{O(d)} \tr(M) \, d\gamma(U) = \tr(M) .
\]
Hence $L$ equals $\frac{1}{d} \tr(M)$ times the identity. 

The second formula in the lemma follows immediately, by multiplying out the matrices and using that $\int_{O(d)} U \, d\gamma(U) = 0$.
\end{proof}

\section{\bf Averaging and spherical homeomorphisms}
\label{sec:Hprop}

Consider a bi-Lipschitz homeomorphism $H : \Sp \to \Sp$, extended by homogeneity to $\Rd \setminus \{ 0 \}$ so that $H(r\xi) = H(\xi)$ for all $r>0, \xi \in \Sp$.
\begin{lemma}[Orthogonality relation] \label{le:orthog}
For all $\xi \in \Sp$, we have $(DH)^\dagger (\xi) H(\xi) = 0$. 
\end{lemma}
\begin{proof}[Proof of \autoref{le:orthog}]
We have $|H(x)|^2 \equiv 1$ for all $x \neq 0$, because $H$ takes values in the unit sphere. Taking the gradient of this last identity yields that $2H(x)^\dagger DH(x) \equiv 0$. Applying the transpose and evaluating at $x=\xi$ completes the proof.
\end{proof}
The lemma implies that $DH(\xi) y \cdot H(\xi)=0$ for all vectors $y$, so that the range of the derivative operator $DH$ at $\xi$ lies orthogonal to the image vector $H(\xi)$, as one would expect since $H$ maps into a sphere.

Some complicated expressions involving $H$ can be simplified considerably, after we average over all orthogonal matrices $U$. We will encounter expressions of the following type when we prove our main theorem.
\begin{proposition} \label{pr:Q23}
Let $H : \Sp \to \Sp$ be a bi-Lipschitz homeomorphism and let $\zeta \in \Sp$ be a fixed unit vector. 

If $F \in L^\infty(\Sp;\Rd)$ is a bounded, row-vector valued function on the sphere, then
\[
\int_{O(d)} F \big( H^{-1}(U\zeta) \big) (DH)^\dagger \big( H^{-1}(U\zeta) \big) U \, d\gamma(U) = 0 .
\]

If $f \in L^\infty(\Sp;\R)$ is a bounded, real-valued function on the sphere, then 
\[
\int_{O(d)} f \big( H^{-1}(U\zeta) \big) \, U^{-1} DH \big( H^{-1}(U\zeta) \big) (DH)^\dagger \big( H^{-1}(U\zeta) \big) U \, d\gamma(U) = c \, (\Id - \zeta \zeta^\dagger) 
\]
where
\[
c = \frac{1}{|\Sp|} \int_\Sp f(\xi) \frac{\lVert DH(\xi) \rVert_{HS}^2}{d-1} \text{Jac}_H(\xi) \, dS(\xi) .
\]
\end{proposition}
\begin{proof}[Proof of \autoref{pr:Q23}]
Choose an orthogonal matrix $W$ that maps $\zeta$ to the north pole, meaning $W \zeta = \nu$ where $\nu=(0,\ldots,0,1)^\dagger$ is the north pole column vector.

We have that
\begin{align}
& \int_{O(d)} F \big( H^{-1}(U\zeta) \big) (DH)^\dagger \big( H^{-1}(U\zeta) \big) U \, d\gamma(U) \label{eq:Q23left} \\
& = \int_{O(d)} F \big( H^{-1}(U\nu) \big) (DH)^\dagger \big( H^{-1}(U\nu) \big) UVW \, d\gamma(U) \label{eq:Q23right}
\end{align}
by changing variable with $U \mapsto UVW$, where $V$ is an arbitrary matrix in $O(d)$ that fixes the north pole (that is, $V\nu=\nu$). We may write
\begin{equation} \label{eq:Vdecomp}
V = \begin{pmatrix} \widetilde{V} & 0 \\ 0 & 1 \end{pmatrix}
\end{equation}
where $\widetilde{V} \in O(d-1)$. 

We are permitted to average expression \eqref{eq:Q23right} with respect to $\widetilde{V}$, since formula \eqref{eq:Q23left} does not depend on $\widetilde{V}$. Notice $V$ appears only once in \eqref{eq:Q23right}. Averaging it yields 
\[
\int_{O(d-1)} V \, d\gamma_{d-1}(\widetilde{V}) = \begin{pmatrix} 0 & 0 \\ 0 & 1 \end{pmatrix}  = \nu \nu^\dagger .
\]
Multiplying on the right by $W$ gives
\[
\int_{O(d-1)} VW \, d\gamma_{d-1}(\widetilde{V})
= \nu \zeta^\dagger
\]
because $W\zeta=\nu$ by construction. Hence after averaging \eqref{eq:Q23right} with respect to $\widetilde{V}$ we obtain that expression \eqref{eq:Q23left} equals
\[
\int_{O(d)} F \big( H^{-1}(U\nu) \big) (DH)^\dagger \big( H^{-1}(U\nu) \big) U \nu \zeta^\dagger \, d\gamma(U) = 0 ,
\]
because
\begin{equation} \label{eq:orthogappl}
(DH)^\dagger \big( H^{-1}(U\nu) \big) U \nu = 0
\end{equation}
by \autoref{le:orthog} applied with $\xi=H^{-1}(U\nu)$.

Now we prove the second formula in the proposition. We find
\begin{align}
& \int_{O(d)} f \big( H^{-1}(U\zeta) \big) \, U^{-1} DH \big( H^{-1}(U\zeta) \big) (DH)^\dagger \big( H^{-1}(U\zeta) \big) U \, d\gamma(U) \label{eq:Q33left} \\
& = \int_{O(d)} f \big( H^{-1}(U\nu) \big) \, W^{-1} (V^{-1} M V)W \, d\gamma(U) \label{eq:Q33right}
\end{align}
by changing variable with $U \mapsto UVW$ as before, where we have defined a matrix
\[
M = U^{-1} DH \big( H^{-1}(U\nu) \big) (DH)^\dagger \big( H^{-1}(U\nu) \big) U .
\]
Note that $M\nu=0$ by \eqref{eq:orthogappl}, and similarly $\nu^\dagger M = 0$ by symmetry of $M$. Hence $M$ has the form
\[
M = \begin{pmatrix} \widetilde{M} & 0 \\ 0 & 0  \end{pmatrix} .
\]
Thus using the decomposition \eqref{eq:Vdecomp} for $V$, we conclude from \autoref{le:averaging} that 
\[
\int_{O(d-1)} V^{-1} M V \, d\gamma_{d-1}(\widetilde{V}) = \frac{1}{d-1} (\tr \widetilde{M}) \begin{pmatrix} \Id_{d-1} & 0 \\ 0 & 0 \end{pmatrix} .
\]

Notice formula \eqref{eq:Q33left} does not depend on $V$. Thus after averaging expression \eqref{eq:Q33right} with respect to $\widetilde{V}$ and using the last formula, we find that expression \eqref{eq:Q33left} equals
\[
\widetilde{c} \ W^{-1} \begin{pmatrix} \Id_{d-1} & 0 \\ 0 & 0 \end{pmatrix} W \quad 
\text{where} \quad \widetilde{c} = \frac{1}{d-1} \int_{O(d)} f \big( H^{-1}(U\nu) \big) (\tr \widetilde{M}) \, d\gamma(U)  .
\]
Since
\[
W^{-1} \begin{pmatrix} \Id_{d-1} & 0 \\ 0 & 0 \end{pmatrix} W
= W^{-1} (\Id - \nu \nu^\dagger) W = \Id - \zeta \zeta^\dagger ,
\]
we deduce that expression \eqref{eq:Q33left} equals $\widetilde{c} \, (\Id - \zeta \zeta^\dagger)$. 

Note that
\begin{align*}
\tr \widetilde{M} = \tr M
& = \tr \Big( DH \big( H^{-1}(U\nu) \big) (DH)^\dagger \big( H^{-1}(U\nu) \big) \Big) \\
& = \lVert DH \big( H^{-1}(U\nu) \big) \rVert_{HS}^2 .
\end{align*}
Hence $\widetilde{c}$ can be evaluated by using the equivalence between orbital and spatial averages (\autoref{app:averages}), which gives that
\[
\widetilde{c} = \frac{1}{d-1} \frac{1}{|\Sp|} \int_\Sp f \big( H^{-1}(\zeta^\prime) \big) \lVert DH \big( H^{-1}(\zeta^\prime) \big) \rVert_{HS}^2 \, dS(\zeta^\prime) .
\]
Lastly, changing variable with $\zeta^\prime=H(\xi)$ shows that $\widetilde{c}$ equals the constant $c$ defined in the Proposition. 
\end{proof}

\section{\bf Dirichlet eigenvalues --- proof of \autoref{Deigen}}
\label{Deigen_proof}

The idea is to obtain trial functions on $\Omega$ by transplanting eigenfunctions from $\B$ to $\Omega$ with the volume-preserving map $T$, and then to average with respect to rotations and reflections of $\B$.

Recall that the Rayleigh quotient associated with the Dirichlet spectrum is
\[
\Ray[u] = \frac{\int_\Omega |\nabla u|^2 \, dx}{\int_\Omega u^2 \, dx} \qquad \text{for\ } u \in H^1_0(\Omega) . 
\]
The Rayleigh--Poincar\'{e} Variational Principle \cite[p.~98]{B80} characterizes the sum of the first $n$ Dirichlet eigenvalues as:
\begin{align*}
\lambda_1 + \dots + \lambda_n & = \min \big\{ \Ray[v_1] + \dots + \Ray[v_n] : \\
& \qquad v_1, \dots,v_n \in H^1_0(\Omega)\text{\ are pairwise orthogonal in $L^2(\Omega)$} \big\} .
\end{align*}

To use this principle, let $u_1,u_2,u_3,\ldots$ be orthonormal eigenfunctions on $\B$ corresponding to the  eigenvalues $\lambda_1(\B),\lambda_2(\B),\lambda_3(\B),\ldots$. Take an orthogonal matrix $U$. Then define trial functions
\[
v_j = u_j \circ U^{-1} \circ T
\]
on the domain $\Omega$, where the transformation $T : \Omega \to \B$ was defined in \autoref{sec:setup} . Clearly $v_j \in L^2(\Omega)$ since $T$ has constant Jacobian. One can further show that $v_j$ has weak derivatives in $L^2(\Omega)$, since $u_j$ is smooth with derivatives in $L^2(\B)$ and the Lipschitz continuous mapping $T$ has bounded weak derivatives. Thus $v_j \in H^1(\Omega)$. Further, $v_j=0$ on $\partial \Omega$ because $u_j=0$ on $\partial \B$; more precisely, $v_j \in H^1_0(\Omega)$ because $u_j \in H^1_0(\B)$.

The functions $v_j$ are pairwise orthogonal, since
\begin{align}
\int_\Omega v_j v_k \, dx 
& = \text{Jac} (T^{-1}) \int_\B u_j u_k \, dx \label{eq:pairorthog} \\
& = 0 \notag
\end{align}
whenever $j \neq k$, using here that $u_j$ and $u_k$ are orthogonal and $T^{-1}$ has constant Jacobian. Thus by the Rayleigh--Poincar\'{e} principle, we have
\begin{equation} \label{eq:rayest}
\sum_{j=1}^n \lambda_j(\Omega) \leq \sum_{j=1}^n \frac{\int_\Omega |\nabla v_j|^2 \, dx}{\int_\Omega v_j^2 \, dx} .
\end{equation}

The denominator of this Rayleigh quotient is $\int_\Omega v_j^2 \, dx = \int_\Omega u_j(U^{-1}T(x))^2 \, dx = \text{Jac}(T^{-1})$ by \eqref{eq:pairorthog} with $j=k$, since the eigenfunctions are normalized with $\int_\B u_j^2 \, dx=1$.

For the numerator of the Rayleigh quotient, we write $v=v_j$ and $u=u_j$ (to simplify notation in what follows) and express $v=v(r,\xi)$ and $u=u(s,\zeta)$ in spherical coordinates. Then the relation $v=u \circ U^{-1} \circ T$ says
\[
v(r,\xi) = u\big( r\Gamma(\xi),U^{-1}H(\xi) \big) ,
\]
by recalling that $\Gamma=1/R$ and using the definition \eqref{eq:Tdef} of $T$. Differentiating, we find
\begin{align*}
v_r(r,\xi) & = \Gamma(\xi) u_s \big( r\Gamma(\xi), U^{-1}H(\xi) \big) , \\
(\nabla_{\! \Sp} v)(r,\xi) & = r u_s \big( r\Gamma(\xi), U^{-1}H(\xi) \big) \nabla \Gamma(\xi) \\
& \qquad + (\nabla_{\! \Sp} u) \big( r\Gamma(\xi) , U^{-1} H(\xi) \big) U^{-1} DH(\xi) ,
\end{align*}
where the gradients are regarded as row vectors and by $\nabla_{\! \Sp}$ we mean the gradient with respect to the angular variables. Writing the numerator of the Rayleigh quotient in terms of spherical coordinates gives that
\[
\int_\Omega |\nabla v|^2 \, dx = \int_\Sp \int_0^{R(\xi)} \big( v_r^2 + r^{-2} |\nabla_{\! \Sp} v|^2 \big) \, r^{d-1} dr dS(\xi) .
\]
After changing variable with $s=r\Gamma(\xi) \in (0,1)$ and using the above formulas for $v_r$ and $\nabla_{\! \Sp} v$, we find that
\[
\int_\Omega |\nabla v|^2 \, dx = Q_1 + Q_2 + Q_3
\]
where
\begin{align*}
Q_1 & = \int_\Sp \int_0^1 \big[ \Gamma(\xi)^2 + |\nabla \Gamma(\xi)|^2 \big] \, u_s \big( s,U^{-1}H(\xi) \big)^{\! 2} \, s^{d-1} ds \, R(\xi)^d dS(\xi) , \\
Q_2 & = 2 \int_\Sp \int_0^1 \Gamma(\xi) \nabla \Gamma (\xi)  (DH)^\dagger(\xi) U \\
& \hspace*{3cm} (\nabla_{\! \Sp} u)^\dagger \big( s,U^{-1}H(\xi) \big) u_s \big( s,U^{-1}H(\xi) \big) \, s^{d-2} ds \, R(\xi)^d dS(\xi) , \\
Q_3 & = \int_\Sp \int_0^1 \Gamma(\xi)^2 \big| (\nabla_{\! \Sp} u)(s,U^{-1}H(\xi)) U^{-1} DH(\xi) \big|^2 \, s^{d-3} ds \, R(\xi)^d dS(\xi) .
\end{align*}

The left side of \eqref{eq:rayest} is independent of $U$. Hence by averaging \eqref{eq:rayest} with respect to $U \in O(d)$ we find
\begin{equation} \label{eq:step0}
\sum_{j=1}^n \lambda_j(\Omega) \leq \sum_{j=1}^n \frac{\int_{O(d)} (Q_1+Q_2+Q_3) \, d\gamma(U)}{\text{Jac}(T^{-1})} ,
\end{equation}
where we must remember that ``$u$'' means $u_j$, in the quantities $Q_1,Q_2,Q_3$.

The quantity $Q_1$ is easiest to average because it contains only one ``$U$''. We have
\[
\int_{O(d)} u_s \big( s,U^{-1}H(\xi) \big)^{\! 2} \, d\gamma(U) = \frac{1}{|\Sp|} \int_\Sp u_s(s,\zeta)^2 \, dS(\zeta)
\]
by the equivalence of orbital and spatial means (\autoref{app:averages}). Hence
\begin{align}
& \frac{\int_{O(d)} Q_1 \, d\gamma(U)}{\text{Jac}(T^{-1})} \notag \\
& = \frac{\frac{1}{|\Sp|} \int_\Sp \big[ \Gamma(\xi)^2 + |\nabla \Gamma(\xi)|^2 \big]  \, R(\xi)^d dS(\xi)}{V(\Omega)/V(\B)} \int_0^1 \int_\Sp u_s(s,\zeta)^2 \, dS(\zeta) \, s^{d-1} ds \notag \\
& = G_0(\Omega)  \Big( \frac{V(\B)}{V(\Omega)} \Big)^{\! \! 2/d} \int_\B \Big( \frac{\partial u_j}{\partial s} \Big)^{\! \! 2} \, dx \label{eq:step1}
\end{align}
by the definition \eqref{eq:G0def} of $G_0$.

For $Q_2$ we begin by changing variable with $\xi=H^{-1}(U\zeta)$, which gives that
\begin{align*}
Q_2 & = \int_\Sp \int_0^1 \Gamma \big( H^{-1}(U\zeta) \big) \nabla \Gamma \big( H^{-1}(U\zeta) \big)  (DH)^\dagger \big( H^{-1}(U\zeta) \big) U \\
& \hspace*{4cm} (\nabla_{\! \Sp} u)^\dagger(s,\zeta) u_s(s,\zeta) \, s^{d-2} ds \, \frac{R \big( H^{-1}(U\zeta) \big)^{\! d}}{\text{Jac}_H \big( H^{-1}(U\zeta) \big)} \, dS(\zeta) .
\end{align*}
The integrand contains $U$ in multiple locations, but averaging remains feasible; indeed \autoref{pr:Q23} applied with $F = (\Gamma \nabla \Gamma) R^d/\text{Jac}_H$ shows that
\begin{equation} \label{eq:step2}
\int_{O(d)} Q_2 \, d\gamma(U) = 0 .
\end{equation}

For $Q_3$ we again change variable with $\xi=H^{-1}(U\zeta)$, and find that
\[
Q_3 = \int_\Sp \int_0^1 \Gamma \big( H^{-1}(U\zeta) \big)^2 \big| \nabla_{\! \Sp} u(s,\zeta) U^{-1} DH \big( H^{-1}(U\zeta) \big) \big|^2 \, s^{d-3} ds  \, \frac{R \big( H^{-1}(U\zeta) \big)^{\! d}}{\text{Jac}_H \big( H^{-1}(U\zeta) \big)} \, dS(\zeta) .
\]
In this integrand $U$ appears five times. Nonetheless, we can average $Q_3$ with respect to $U$ by expanding $|\cdots|^2$ and using \autoref{pr:Q23} with $f=\Gamma^2 R^d/\text{Jac}_H$. We find that
\begin{equation} \label{eq:Q3average}
\int_{O(d)} Q_3 \, d\gamma(U) = c \int_\Sp \int_0^1 (\nabla_{\! \Sp} u) (\Id - \zeta \zeta^\dagger)  (\nabla_{\! \Sp} u)^\dagger \, s^{d-3} ds dS(\zeta) 
\end{equation}
where 
\[
c = \frac{1}{|\Sp|} \int_\Sp \Gamma(\xi)^2 \frac{\lVert DH(\xi) \rVert_{HS}^2}{d-1} R(\xi)^d \, dS(\xi) .
\]
[\emph{Aside.} The averaging results \eqref{eq:step2} and \eqref{eq:Q3average} for $Q_2$ and $Q_3$ hold independently of the specific form of the Jacobian of $H$, provided the Jacobian is bounded away from zero (since we want to avoid trouble when we divide by it).]

Note that $(\nabla_{\! \Sp} u) \zeta = 0$, because the spherical gradient $\nabla_{\! \Sp} u$ lies perpendicular to the unit vector $\zeta$. Hence we deduce that
\begin{align}
\frac{\int_{O(d)} Q_3 \, d\gamma(U)}{\text{Jac}(T^{-1})}
& = \frac{V(\B)}{V(\Omega)} \, c \int_\Sp \int_0^1 (\nabla_{\! \Sp} u)(\nabla_{\! \Sp} u)^\dagger \, s^{d-3} ds dS(\zeta) \notag \\
& = G_1(\Omega)  \Big( \frac{V(\B)}{V(\Omega)} \Big)^{\! \! 2/d} \int_\B s^{-2} |\nabla_{\! \Sp} u_j|^2 \, dx \label{eq:step3}
\end{align}
by definition of $G_1$ in \eqref{eq:G1def}. 

Combining \eqref{eq:step0}--\eqref{eq:step3} now shows that
\begin{align}
\sum_{j=1}^n \lambda_j(\Omega)
& \leq \Big( \frac{V(\B)}{V(\Omega)} \Big)^{\! \! 2/d} \sum_{j=1}^n \Big[ G_0(\Omega) \int_\B \big( \frac{\partial u_j}{\partial s} \big)^2 \, dx + G_1(\Omega) \int_\B s^{-2} |\nabla_{\! \Sp} u_j|^2 \, dx \Big] \notag \\
& = \Big( \frac{V(\B)}{V(\Omega)} \Big)^{\! \! 2/d} \sum_{j=1}^n \Big[ (1-\alpha_j) G_0(\Omega) + \alpha_j G_1(\Omega) \Big] \int_\B |\nabla u_j|^2 \, dx \label{eq:partway}
\end{align}
where 
\[
\alpha_j = \frac{\int_\B s^{-2} |\nabla_{\! \Sp} u_j|^2 \, dx}{\int_\B |\nabla u_j|^2 \, dx} , \qquad j=1,2,\ldots,n .
\]
(The coefficient $\alpha_j \in [0,1]$ measures the ``angular component'' of the $j$th energy.) Next we estimate $G_0$ and $G_1$ from above with their maximum $G$, and so conclude that
\[
\sum_{j=1}^n \lambda_j(\Omega) V(\Omega)^{2/d}/G(\Omega) \leq V(\B)^{2/d} \sum_{j=1}^n \int_\B |\nabla u_j|^2 \, dx = \sum_{j=1}^n \lambda_j(\B) V(\B)^{2/d} .
\]
Since $G(\B)=1$, we have proved the theorem in the case that $\Phi(a) \equiv a$ is the identity function.

The theorem now follows for any concave increasing $\Phi$, by Hardy--Littlewood--P\'{o}lya majorization as in Appendix~\ref{sec:major}.

Notice that for the first eigenvalue our proof gives a stronger conclusion, namely using $G_0$ instead of $G$, because the fundamental mode $u_1$ of the ball is a radial function and so $\alpha_1=0$ in the argument above. Thus $\lambda_1 V^{2/d}/G_0$ is maximal for the ball.

For the second Dirichlet eigenvalue, we may divide and multiply by the first eigenvalue to obtain that
\[
\frac{\lambda_2 V^{2/d}}{G_0} = \frac{\lambda_1 V^{2/d}}{G_0} \, \frac{\lambda_2}{\lambda_1}  .
\]
We previously showed that the first factor is maximal for the ball, and the second factor is maximal too, by the sharp Payne--P\'{o}lya--Weinberger result of Ashbaugh and Benguria \cite{AB92}.

\subsection*{Particular cases.} Applying the theorem with $\Phi(a)=a^s$, which is concave and increasing when $0 < s \leq 1$, gives maximality of $(\lambda_1^s + \cdots + \lambda_n^s)^{1/s} \, V^{2/d}/G$ for the ball. The limiting case $s \downarrow 0$ suggests we try choosing $\Phi(a)=\log a$, which yields maximality of the ball for the functional
\[
\sum_{j=1}^n \log (\lambda_j V^{2/d} /G) = n \log \Big( \sqrt[n]{\lambda_1 \cdots \lambda_n} \, V^{2/d}/G \Big) .
\]
When $s<0$ we can choose the concave increasing function $\Phi(a)= - a^s$, which leads to minimality of the ball for
$\sum_{j=1}^n (\lambda_j V^{2/d} /G)^s$. And for $t>0$ we can consider $\Phi(a)=-e^{-at}$, thus obtaining minimality at the ball of $\sum_{j=1}^n \exp(-\lambda_j V^{2/d} t/G)$.

\subsection*{Dirichlet equality statement.} Assume equality holds for the first eigenvalue, that is,
\[
\lambda_1 V^{2/d} /G_0 \big|_\Omega = \lambda_1 V^{2/d} /G_0 \big|_\B .
\]
By enforcing equality in our proof above, we see that the trial function $v_1$ on $\Omega$ must attain equality in the Rayleigh characterization of $\lambda_1(\Omega)$, and hence must be a first eigenfunction for $\Omega$. In particular this holds when the orthogonal matrix $U$ is the identity, so that the function $v_1(x)=u_1(T(x))$ satisfies
\begin{equation} \label{eq:equaleigen}
\Delta v_1 = -\lambda_1(\Omega) v_1 .
\end{equation}

The fundamental Dirichlet mode $u_1$ of the ball is radial, with $u_1(x)=J(|x|)$ for some positive function $J$, and so we have $v_1(x)=J(r/R(\xi))$. That is,
\[
v_1(x) = J(r\Gamma(\xi))
\]
where $\Gamma=1/R$. Note $R$ is $C^2$-smooth by assumption, in this part of the theorem.

The Laplacian of $v_1$ is given in spherical coordinates by
\[
\Delta v_1(x) = J^{\prime \prime}\big(r\Gamma(\xi)\big) \Gamma(\xi)^2 + \frac{d-1}{r} J^\prime\big(r\Gamma(\xi)\big) \Gamma(\xi) + \frac{1}{r^2} \Delta_\Sp [J\big(r\Gamma(\xi)\big)] .
\]
This spherical Laplacian can be computed by the chain rule. It equals
\[
\Delta_\Sp [J\big(r\Gamma(\xi)\big)] = J^{\prime \prime}\big(r\Gamma(\xi)\big) r^2 |\nabla \Gamma(\xi)|^2 + J^\prime\big(r\Gamma(\xi)\big) r \Delta_\Sp \Gamma(\xi) .
\]
We substitute this formula into the preceding one, and make the substitution $s=r\Gamma(\xi)$. Then the eigenfunction equation \eqref{eq:equaleigen} reads:
\[
\big[ \Gamma(\xi)^2 + |\nabla \Gamma(\xi)|^2 \big] J^{\prime \prime}(s) + [(d-1)\Gamma(\xi)^2 + \Gamma(\xi) \Delta_\Sp \Gamma(\xi)] \frac{1}{s} J^\prime(s) = - \lambda_1(\Omega) J(s) .
\]
Integrating over $\xi \in \Sp$ yields that
\[
\lVert \Gamma \rVert_{H^1}^2 J^{\prime \prime}(s) + [(d-1) \lVert \Gamma \rVert_2^2 - \lVert \nabla \Gamma \rVert_2^2] \, \frac{1}{s} J^\prime(s) = - \lambda_1(\Omega) |\Sp| J(s) .
\]
The eigenfunction equation for the unit ball (the case $\Gamma \equiv 1$) says that
\begin{equation} \label{eq:eigenradial}
J^{\prime \prime}(s) + \frac{d-1}{s} J^\prime(s) = - \lambda_1(\B) J(s) .
\end{equation}
We subtract $\lVert \Gamma \rVert_{H^1}^2$ times this equation from the previous equation,
thereby obtaining a first order equation for $J$:
\[
\lVert \nabla \Gamma \rVert_2^2 J^\prime(s) = \frac{1}{d} \big( \lambda_1(\Omega) |\Sp| - \lambda_1(\B) \lVert \Gamma \rVert_{H^1}^2 \big) \, s J(s) , \qquad 0<s<1 .
\]

Suppose $\Gamma \not \equiv \text{const.}$, which ensures that $\lVert \nabla \Gamma \rVert_2 > 0$. Then the last equation for $J$ has the form
\[
J^\prime(s) = asJ(s)
\]
for some $a \in \R$, and so $J^{\prime \prime}(s) = (a^2s^2+a)J(s)$. Substituting these relations into the eigenfunction equation \eqref{eq:eigenradial} implies $a^2s^2+ad=-\lambda_1(\B)$ for all $s \in (0,1)$, and so $a=0$ and then $\lambda_1(\B)=0$. This contradiction tells us that $\Gamma \equiv \text{const.}$, and so $R \equiv \text{const.}$, which means that $\Omega$ is a centered ball.

\section{\bf Neumann eigenvalues --- proof of \autoref{Neigen}}
\label{Neigen_proof}

For Neumann boundary conditions, the Rayleigh quotient and Rayleigh--Poincar\'{e} Principle are just as for the Dirichlet case, except using trial functions in $H^1(\Omega)$ rather than $H^1_0(\Omega)$. Thus we may follow the proof of \autoref{Deigen}, except using Neumann eigenfunctions of the ball instead of Dirichlet eigenfunctions, to prove that
\[
\sum_{j=1}^n \mu_j(\Omega) V(\Omega)^{2/d}/G(\Omega) \leq \sum_{j=1}^n \mu_j(\B) V(\B)^{2/d} .
\]
On each side, the term with $j=1$ may now be discarded because $\mu_1=0$. Then the proof can be completed by majorization.

For the equality statement, rather than adapting the Dirichlet case we present a simpler approach. Suppose equality holds for the first nonzero eigenvalue, that is,
\begin{equation} \label{eq:neumanequal}
\mu_2 V^{2/d} /G \big|_\Omega = \mu_2 V^{2/d} /G \big|_\B .
\end{equation}
Since $\mu_2 V^{2/d} \big|_\Omega \leq \mu_2 V^{2/d} \big|_\B$ by the Szeg\H{o}--Wein\-berger result \cite{W56} (or see \cite[Theorem 7.1.1]{He06}), and since $G(\Omega) \geq 1 = G(\B)$, we conclude from equality holding in \eqref{eq:neumanequal} that $G(\Omega)=1$. Hence $G_0(\Omega)=1$, and so $\Omega$ is a centered ball by the equality statement in \autoref{le:leq1}.

\section{\bf Robin eigenvalues --- proof of \autoref{Reigen}}
\label{Reigen_proof}

The Rayleigh quotient for the Robin problem is
\begin{equation} \label{eq:robinraydef}
\Ray[u] = \frac{\hbar^2 \int_\Omega |\nabla u|^2 \, dx + \int_{\partial \Omega} \sigma u^2 \, dS}{\int_\Omega u^2 \, dx} \qquad \text{for\ } u \in H^1(\Omega) . 
\end{equation}

Let $\{ u_j \}$ be orthonormal eigenfunctions on the unit ball $\B$ that correspond to the Robin eigenvalues $\rho_j \big(\B,\hbar V(\B)^{1/d},\overline{\sigma} V(\B)^{1/d} \big)$, for $j=1,2,3,\ldots$. By constructing trial functions and using the Rayleigh--Poincar\'{e} principle as in the Dirichlet case (\autoref{Deigen_proof}), we find the following analogue of \eqref{eq:rayest}:
\begin{align}
& \sum_{j=1}^n \rho_j \big(\Omega,\hbar V^{1/d} /G^{1/2},\sigma V^{1/d} /G_\Robin^{1/2} \big) \notag \\
& \leq \sum_{j=1}^n \frac{\hbar^2 V^{2/d}}{G} \, \frac{\int_\Omega |\nabla v_j|^2 \, dx}{\int_\Omega v_j^2 \, dx} + \sum_{j=1}^n \frac{V^{1/d}}{G_\Robin^{1/2}} \, \frac{\int_{\partial \Omega} \sigma v_j^2 \, dS}{\int_\Omega v_j^2 \, dx} , \label{eq:RobinRayleigh}
\end{align}
where $v_j=u_j \circ U^{-1} \circ T$. Averaging over $U \in O(d)$ (as explained in \autoref{Deigen_proof} leading up to \eqref{eq:partway}) shows that the first sum in \eqref{eq:RobinRayleigh} is bounded from above by
\begin{equation} \label{eq:firstterms}
\sum_{j=1}^n \hbar^2 V(\B)^{2/d} \int_\B |\nabla u_j|^2 \, dx .
\end{equation}
For the principal eigenvalue ($n=1$) this part of the argument also works with $G_0$ in place of $G$, since $u_1$ is  radial. (When finding $u_1$ by separation of variables, the spherical harmonics with angular dependence cannot arise, because $u_1$ is positive. Hence $u_1$ is radial.)

We will show below that averaging the second sum in \eqref{eq:RobinRayleigh} gives
\begin{equation} \label{eq:secondterms}
\sum_{j=1}^n \overline{\sigma} V(\B)^{1/d} \int_\Sp u_j^2 \, dS .
\end{equation}
The theorem then follows, because adding \eqref{eq:firstterms} and \eqref{eq:secondterms} gives  
\[
\sum_{j=1}^n \rho_j \big(\B,\hbar V(\B)^{1/d},\overline{\sigma} V(\B)^{1/d} \big) .
\]

For \eqref{eq:secondterms} it suffices to consider one value of $j$ at a time, and so we consider an arbitrary function $u \in H^1(\Omega)$ with $L^2$-norm equal to $1$, and write $v=u \circ U^{-1} \circ T$. The second sum in \eqref{eq:RobinRayleigh} has terms of the form
\begin{equation} \label{eq:Robin1}
\frac{V^{1/d}}{G_\Robin^{1/2}} \, \frac{\int_{\partial \Omega} \sigma v^2 \, dS}{\int_\Omega v^2 \, dx}
=  \frac{V^{1/d}}{G_\Robin^{1/2}} \, \frac{\int_{\partial \Omega} \sigma u \big( U^{-1}T(x) \big)^2 \, dS(x)}{V(\Omega)/V(\B)} ,
\end{equation}
where in the denominator we changed variable and used that $T$ has constant Jacobian and $u$ has $L^2$-norm equal to $1$. Averaging the right side of \eqref{eq:Robin1} over matrices $U \in O(d)$ gives (by the equivalence of orbital and spatial means, as in \autoref{app:averages}) the expression
\[
\frac{V(\Omega)^{1/d}}{G_\Robin(\Omega)^{1/2}} \frac{V(\B)}{V(\Omega)}\Big(  \int_{\partial \Omega} \sigma \, dS(x) \Big) \Big( \frac{1}{|\Sp|} \int_\Sp u^2 \, dS \Big) .
\]
This last expression equals $V(\B)^{1/d} \, \overline{\sigma} \int_\Sp u^2 \, dS$ by definition of $G_\Robin$ and $\overline{\sigma}$, proving \eqref{eq:secondterms}.

To prove inequality \eqref{eq:Robinfirst} for the first eigenvalue, apply the theorem with $n=1$ and $G_0$ instead of $G$ (as remarked above), and replace $\hbar$ by $\hbar/V(\B)^{1/d}$ and replace $\sigma$ by $\sigma/V(\B)^{1/d}$.

For the equality statement on the first eigenvalue, one simply adapts the proof of the Dirichlet equality statement in \autoref{Deigen_proof}.

\section{\bf Improvement to the main results}
\label{sec:improvements}

Our main theorems attach the geometric factor $G=\max\{ G_0,G_1 \}$ to each eigenvalue. An inspection of the proofs yields a stronger result, in which each eigenvalue is paired with a smaller geometric factor arising from a convex combination of $G_0$ and $G_1$. We state this improved result below, restricting for simplicity to the case of eigenvalue sums. (The reader can then deduce inequalities on spectral zeta functions and so on, by applying the majorization result from Appendix~\ref{sec:major}.) To simplify the exposition we do not treat the Robin case.

Define a convex combination of the geometric factors by
\[
G(\alpha;\Omega) = (1-\alpha) G_0(\Omega) + \alpha G_1(\Omega) , \qquad \alpha \in [0,1] .
\]

To choose the relevant values of $\alpha$, we fix an orthonormal basis of eigenfunctions $u_1,u_2,u_3,\ldots$ of the unit ball $\B$ corresponding to the Dirichlet eigenvalues $\lambda_1(\B), \lambda_2(\B), \lambda_3(\B), \cdots $. The \emph{radial energy fraction} of the $j$th Dirichlet eigenfunction is defined to be
\[
\e^D_j = \frac{\int_\B (\partial u_j/\partial s)^2 \, dx}{\int_\B |\nabla u_j|^2 \, dx} ,
\]
where $s \in [0,1]$ denotes the radial variable. This energy fraction can be computed explicitly by writing $u_j$ in terms of Bessel functions. Obviously $0 < \e^D_j \leq 1$, with $\e^D_j=1$ if and only if $u_j$ is purely radial.

The \emph{angular energy fraction} is then
\[
\alpha^D_j = 1 - \e^D_j .
\]
For example, the principal Dirichlet mode of the ball is radial, and so $\alpha^D_1=0$.

Similarly, we may define the angular energy fraction $\alpha^N_j$ for the Neumann eigenfunctions of the unit ball.
\begin{theorem}[Improved inequalities] \label{improved}
Assume $\Omega$ is a Lipschitz-starlike domain in $\Rd$. Then the Dirichlet and Neumann eigenvalues satisfy
\begin{align*}
\sum_{j=1}^n \lambda_j(\Omega) V(\Omega)^{2/d} & \leq \sum_{j=1}^n \lambda_j(\B) V(\B)^{2/d} G(\alpha^D_j;\Omega) , \qquad n \geq 1 ,\\
\sum_{j=2}^n \mu_j(\Omega) V(\Omega)^{2/d} & \leq \sum_{j=2}^n \mu_j(\B) V(\B)^{2/d} G(\alpha^N_j;\Omega) , \qquad n \geq 2 .
\end{align*}
\end{theorem}
\begin{proof}[Proof of \autoref{improved}]
See \eqref{eq:partway} in the proof of \autoref{Deigen}. The Neumann case is analogous.
\end{proof}

\section{\bf Properties of the geometric factors}
\label{sec:understanding}

\subsection*{$G_1$ does not depend on $H$ in $2$ dimensions}
The quantity $G_1$ defined in \eqref{eq:G1def} depends only on $R$ and not on $H$, in $2$ dimensions, by the following result.
\begin{proposition} \label{le:G2dim}
In dimension $d=2$,
\begin{align}
G_0 & = 1 + \frac{1}{2\pi} \int_0^{2\pi} (\log R)^\prime(\theta)^2 \, d\theta , \label{eq:G2dim1} \\
G_1 & = \frac{\frac{1}{2\pi} \int_0^{2\pi} R(\theta)^4 \, d\theta}{\big( \frac{1}{2\pi} \int_0^{2\pi} R(\theta)^2 \, d\theta \big)^2 } = \frac{2\pi I_\text{origin}}{A^2} , \label{eq:G2dim2}
\end{align}
where $A$ is the area of $\Omega$ and $I_\text{origin}=\int_\Omega |x|^2 \, dA$ is its polar moment of inertia about the origin. Further,
\[
G_\Robin = \frac{L^2}{4\pi A}
\]
where $L$ is the perimeter of $\Omega$.
\end{proposition}
These formulas imply immediately that $G_0,G_1,G_\Robin$ are $\geq 1$ and are scale invariant with respect to dilations of the domain, in $2$ dimensions.
\begin{proof}[Proof of \autoref{le:G2dim}]
To prove the first and third formulas, simply substitute $d=2$ into the definitions \eqref{eq:G0def} and \eqref{eq:GRobindef} of $G_0$ and $G_\Robin$.

For the second formula, when $d=2$ the definition \eqref{eq:G1def} of $G_1$ implies
\begin{equation}
\label{eq:c3part}
G_1 = \int_0^{2\pi} \lVert DH( \begin{smallmatrix} \cos \theta \\ \sin \theta \end{smallmatrix} ) \rVert_{HS}^2 \, d\theta \big/ 2\pi .
\end{equation}
The homeomorphism $H$ of the unit circle can be written $H(\begin{smallmatrix} \cos \theta \\ \sin \theta \end{smallmatrix})=(\begin{smallmatrix} \cos \phi(\theta) \\ \sin \phi(\theta) \end{smallmatrix})$. Homogeneity of $H$ then gives $H(\begin{smallmatrix} x_1 \\ x_2 \end{smallmatrix}) = ( \begin{smallmatrix} \cos \phi(\theta) \\ \sin \phi(\theta) \end{smallmatrix} )$ where $\theta=\arg(x_1+ix_2)=\arctan(x_2/x_1)$. Calculating the derivative matrix $DH(\begin{smallmatrix} x_1 \\ x_2 \end{smallmatrix})$ results in
\[
\lVert DH( \begin{smallmatrix} \cos \theta \\ \sin \theta \end{smallmatrix} ) \rVert_{HS}^2 = \phi^\prime(\theta)^2 .
\]
Further, the distortion formula \eqref{eq:distort} for $H$ says in $2$ dimensions that
\[
\phi^\prime(\theta)^2 = \Big( \frac{A(\B)}{A(\Omega)} R(\theta)^2 \Big)^{\! 2}.
\]
Substituting the last two formulas into \eqref{eq:c3part} shows that
\[
G_1 = \frac{\int_0^{2\pi} R(\theta)^4 \, d\theta \big/ 2\pi}{\big(A(\Omega)/\pi \big)^2} .
\]
Now \eqref{eq:G2dim2} follows by evaluating area and moment of inertia in polar coordinates.
\end{proof}

\subsection*{Expressing $G_0$ as a support-type functional, in all dimensions}
The geometric meaning of $G_0$ is highlighted by:
\begin{lemma}[Equivalence of our $G_0$ with the definitions of P\'{o}lya--Szeg\H{o} and Freitas--Krej{\v{c}}i{\v{r}}{\'{\i}}k] \label{le:alt}
If $\Omega$ is a Lipschitz-starlike domain then
\[
G_0 = \frac{1}{ |\Sp|} \int_{\partial \Omega} \frac{1}{x \cdot N(x)} \, dS(x) \Big( \frac{V(\B)}{V(\Omega)} \Big)^{\! (d-2)/d}.
\]
Thus in $2$ dimensions,
\begin{equation} \label{eq:equiv2dim}
G_0 = \frac{1}{2\pi} \int_{\partial \Omega} \frac{1}{x \cdot N(x)} \, ds(x) .
\end{equation}

Hence if $\Omega$ is convex then $G_0 \leq (V/V_{in})^{2/d}$, where $V_{in}$ is the volume of the largest open ball centered at the origin and contained in $\Omega$. .
\end{lemma}
P\'{o}lya and Szeg\H{o}'s calculations already prove the lemma in dimension $2$ (see \cite[p.~92]{PS51}), but the lemma is new in higher dimensions because Freitas and Krej{\v{c}}i{\v{r}}{\'{\i}}k proceeded along somewhat different lines in their proof.
\begin{proof}[Proof of \autoref{le:alt}]
Our first task is to evaluate $x \cdot N(x)$ in terms of the radius function. The boundary of $\Omega$ is the level set $\{ x : |x|^2=R(x)^2 \}$ and so taking the gradient gives an outward normal vector $n(x) = x-R(x)\nabla R(x)$. We evaluate at $x=R(\xi)\xi \in \partial \Omega$ to obtain $n(x) = R(\xi)\xi-R(\xi) \nabla R \big(R(\xi)\xi \big)$. The homogeneity relation $R(r\xi)=R(\xi)$ implies that $r\nabla R(r\xi)=\nabla R(\xi)$ for each $r>0$, and so $n(x) = R(\xi)\xi-\nabla R(\xi)$. Thus for the unit normal $N(x)=n(x)/|n(x)|$ we compute
\[
x \cdot N(x) = \frac{R(\xi)^2}{\sqrt{R(\xi)^2+|\nabla R(\xi)|^2}} .
\]
where we used that $\xi \cdot \nabla R(\xi) = 0$ (by homogeneity of $R$).

Next we need a formula for surface area element on the boundary of $\Omega$:
\[
dS(x) = R(\xi)^{d-2} \sqrt{R(\xi)^2+|\nabla R(\xi)|^2} \, dS(\xi) ,
\]
as one proves straightforwardly by parameterizing $\partial \Omega$ as $\{ x=R(\xi) \xi : \xi \in \Sp \}$.

By substituting the preceding formulas into the formula for $G_0$ in the lemma, we see that it reduces to the definition of $G_0$ in \autoref{sec:setup}.

Now write $\B_{in}$ for the ball of volume $V_{in}$ centered at the origin, and write $R_{in}$ for its radius. If $\Omega$ is convex then $R_{in} \leq x \cdot N(x)$ for all $x \in \Omega$, as one sees by considering a support plane at $x$, and so
\[
\frac{1}{x \cdot N(x)} \leq \frac{x \cdot N(x)}{R_{in}^2} .
\] 
By integrating over $\partial \Omega$ and using the formula in the lemma for $G_0$, on the left side, and the divergence theorem on the right side, we find that $G_0 \leq (V/V_{in})^{2/d}$. 
\end{proof}

\subsection*{Evaluation of $G_0$ for polygons with an inscribed circle}
If $\Omega$ is a triangle with incenter at the origin, or more generally if $\Omega$ is any polygon with an inscribed circle centered at the origin, then the geometric factor $G_0$ can be evaluated in terms of area and perimeter. For such domains $G_0=L^2/4\pi A=G_\Robin$, as was proved by Aissen \cite[Theorem~1]{A58}. To prove this fact observe that $x \cdot N(x)$ equals the inradius, for each point $x$ on the boundary. Hence formula \eqref{eq:equiv2dim} gives that $G_0$ equals $L/2\pi$ divided by the inradius, which evaluates to $L^2/4\pi A$ because $A=\frac{1}{2} L R_{in}$ (by triangulating the domain with respect to the origin).

\subsection*{Proof that $G_0 \geq G_\Robin \geq 1$ and $G_1 \geq 1$ (for Lemmas~\ref{le:leq1} and \ref{le:leq1Robin})} \

We have from \autoref{le:alt} and Cauchy--Schwarz that
\begin{align}
G_0
& \geq \frac{1}{|\Sp|} \frac{|\partial \Omega|^2}{\int_{\partial \Omega} x \cdot N(x) \, dS(x)} \Big( \frac{V(\B)}{V(\Omega)} \Big)^{\! (d-2)/d} \label{eq:AissenCS} \\
& = \frac{1}{|\Sp|} \frac{|\partial \Omega|^2}{V(\Omega) d} \Big( \frac{V(\B)}{V(\Omega)} \Big)^{\! (d-2)/d}  = G_\Robin \notag
\end{align}
by the divergence theorem (noting $\nabla \cdot x \equiv d$), and by definition of $G_\Robin$ in \eqref{eq:GRobindef}. Note that if equality holds then $x \cdot N(x)$ is constant, by the equality conditions for Cauchy--Schwarz.

Further, $G_\Robin \geq 1$ by the isoperimetric inequality, as remarked after \eqref{eq:GRobindef}, with equality if and only if $\Omega$ is a ball.

Hence $G_0 \geq G_\Robin \geq 1$, and $G_0=1$ if and only if $\Omega$ is a centered ball.

(\emph{Aside.} This inequality was established in $2$ dimensions by Aissen \cite[Theorem~1]{A58}.)

Now we prove $G_1 \geq 1$. By applying the quadratic-geometric mean inequality to the nonzero singular values of $DH$ we deduce that
\[
\frac{\lVert DH(\xi) \rVert_{HS}^2}{d-1}
\geq \text{Jac}_H(\xi)^{\! 2/(d-1)}
= \Big( \frac{V(\B)}{V(\Omega)} \Big)^{2/(d-1)} R(\xi)^{2d/(d-1)} ,
\]
where the last step uses the distortion formula for $H$ in \eqref{eq:distort}. Substituting this estimate into the definition of $G_1$ in \eqref{eq:G1def} shows that
\begin{align*}
G_1
& \geq \Big( \frac{V(\B)}{V(\Omega)} \Big)^{2/(d-1)} \Big( \frac{1}{|\Sp|} \int_\Sp (R^d)^{1+2/d(d-1)} \, dS \Big) \Big/ \Big( \frac{1}{|\Sp|} \int_\Sp R^d \, dS \Big)^{(d-2)/d} \\
& \geq \Big( \frac{V(\B)}{V(\Omega)} \Big)^{2/(d-1)} \Big( \frac{1}{|\Sp|} \int_\Sp R^d \, dS \Big)^{\! 2/(d-1)}
\end{align*}
by Jensen's inequality. Since $\int_\Sp R^d \, dS = V(\Omega)d$ and $|\Sp| = V(\B)d$, we deduce that $G_1 \geq 1$.

If $G_1=1$ then $R$ is constant (by the equality conditions for Jensen), and so $\Omega$ is a centered ball. Then $H$ maps $\Sp$ to $\Sp$ with constant distortion $\text{Jac}_H \equiv 1$, by \eqref{eq:distort}. In $2$ dimensions that is enough to imply $H$ is an orthogonal transformation (cf.~\autoref{sec:Hexist}). So suppose $d \geq 3$. Note the nonzero singular values of $DH$ all have equal magnitude (and hence have magnitude $1$), by the equality conditions for the quadratic-geometric mean inequality. Therefore by the singular value decomposition, $DH$ acts as an orthogonal matrix on the tangent space, at almost every point of the sphere. Liouville's theorem \cite{IM01} implies that $H$ is a M\"{o}bius transformation that fixes the sphere. Since also the Jacobian of $H$ equals $1$ identically on the sphere, we conclude that $H$ is an orthogonal transformation.

\subsection*{Good choices of origin}
What is a good choice of origin within the domain, given that for \autoref{Deigen} we would like to make the geometric factors $G_0$ and $G_1$ as small as possible?

To minimize $G_1$ one should choose the origin at the center of mass, because \autoref{le:G2dim} expresses $G_1$ in terms of moment of inertia (at least in $2$ dimensions).

The center of mass is not generally the best choice of origin for $G_0$. For a polygon with an inscribed circle, $G_0$ is minimal when the origin coincides with the center of the circle; this observation is due to Aissen \cite[\S3]{A58}, with the key step being the use of Cauchy--Schwarz as in \eqref{eq:AissenCS}. For a triangle, for example, the inscribed circle is centered where the angle bisectors intersect, which can be quite far from the center of mass (as happens for a thin acute isosceles triangle). Thus in general one cannot hope to minimize both $G_0$ and $G_1$ with a single choice of origin. An exception is for domains having two axes of symmetry, in which case both factors are minimized when the origin is at the intersection (for $G_1$ because the intersection point is the centroid, and for $G_0$ by work of Aissen \cite[Cor.~1,2,3]{A58}).

The question of whether $G_0$ or $G_1$ is larger can be subtle to resolve. For example, consider the ellipse with semi-axes $3$ and $1$, shown in \autoref{fig1}. Choosing the origin at the center would minimize both factors, and in fact would make them equal (as one finds by direct computation). Nearby choices of origin, though, could lead to either $G_0$ or $G_1$ being larger. Thus for domains like perturbed ellipses it is unclear which factor will dominate, until computations have been performed. 

\begin{figure}[t]
  \begin{center}
\includegraphics[scale=0.235]{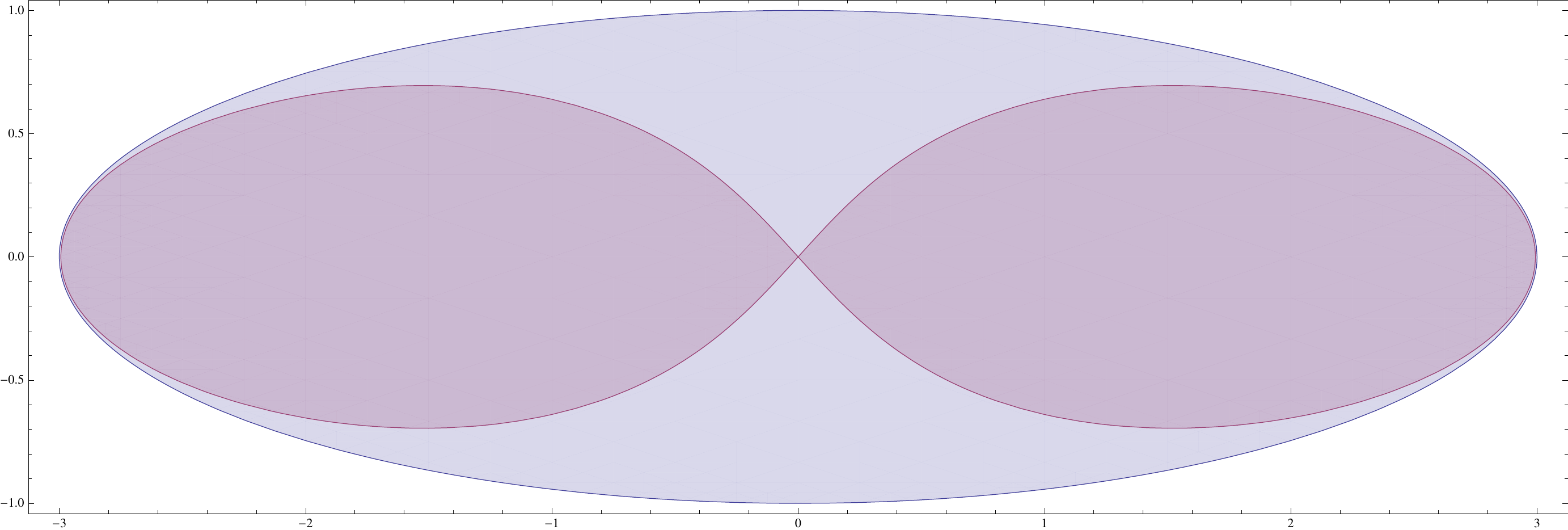}
  \end{center}
  \caption{Choices of origin in the dark shaded region give $G_0<G_1$.}
  \label{fig1}
\end{figure}

\section{\bf Existence of homeomorphism $H$ --- Proof of \autoref{homexist}}
\label{sec:Hexist}

Here we construct a bi-Lipschitz homeomorphism $H : \Sp \to \Sp$ with specified Jacobian determinant, as required in \autoref{sec:setup}. We apply the construction to ellipsoids in the next section, in order to better understand the geometric factor $G_1$.

Write $K(\xi) = R(\xi)^d V(\B)/V(\Omega)$ for the desired Jacobian determinant of $H$, which one regards as the mass density for the $H$-pullback of the uniform mass density on the sphere. The following construction of $H$ remains valid whenever $K$ is continuous and positive with $\int_\Sp K(\xi) \, d\xi = |\Sp|$; the specific form of $K$ is irrelevant. 

\subsection*{Two dimensions, $d=2$}
In two dimensions we may regard $K$ as a $2\pi$-periodic function of an angle $\theta$. The Jacobian condition ${\text Jac}_H=K$ says $H^\prime(\theta) = K(\theta)$, which we satisfy by defining
\[
H(\theta) = \int_0^\theta K(\omega) \, d\omega .
\]
Notice $H$ increases by $2\pi$ each time $\theta$ increases by $2\pi$, since $\int_0^{2\pi} K(\omega) \, d\omega = |{\mathbb S}^1|=2\pi$. Also, $H^\prime=K$ is continuous, and is bounded above and below away from $0$. Hence $H$ defines a $C^1$-diffeomorphism of the circle.

Incidentally, this construction shows that $H$ is uniquely determined on the circle, except for post-rotations (adding a constant to $H$). We need not consider reflections because $H$ has positive Jacobian and so it must be orientation preserving.

\subsection*{Three dimensions, $d=3$: the latitude--longitude construction}
Let $(\theta_1,\theta_2)$ be the standard spherical coordinates, with $0 \leq \theta_1 \leq \pi$ and $0 \leq \theta_2 \leq 2\pi$. We assume $H$ has the form 
\[
 H(\theta_1,\theta_2)  = \big (f(\theta_1),g(\theta_1,\theta_2) \big) ,
\]
which means that each line of latitude ($\theta_1=\text{const.}$) is mapped to another line of latitude, and the longitudinal position is transformed by $g$. We will first determine $f$ by studying how the spacing between lines of latitude must be distorted, and then will determine $g$ by applying the earlier $2$-dimensional method on each line of latitude.

Fix the north and south poles (meaning $f(0)=0,f(\pi)=\pi$) and require that $g$ increase by $2\pi$ for each trip around a line of latitude (meaning $g(\cdot,\cdot+2\pi)-g=2\pi$). Note that 
\begin{align}
-(\cos f(\theta_1))^\prime g_{\theta_2}(\theta_1,\theta_2) 
& = \big( \sin f(\theta_1) \big) f^\prime(\theta_1) g_{\theta_2}(\theta_1,\theta_2) \notag \\
& = K(\theta_1,\theta_2) \sin \theta_1 , \label{eq:theta2eq}
\end{align}
by the Jacobian condition ${\text Jac}_H=K$. Integrating over $\theta_2 \in [0,2\pi]$ gives an equation involving only $\theta_1$:
\begin{equation} \label{eq:theta1eq}
-(\cos f(\theta_1))^\prime = \Big( \frac{1}{2\pi} \int_0^{2\pi} K(\theta_1,\theta_2) \, d\theta_2 \Big) \sin \theta_1 .
\end{equation}
This equation can be solved for $\cos f(\theta_1)$ by direct integration, using the north pole condition $\cos f(0)=1$. (The south pole condition $\cos f(\pi)=-1$ then follows automatically, since we have $\int_0^\pi \int_0^{2\pi} K \sin \theta_1 \, d\theta_2 d\theta_1 = \int_{\Sp^2} K \, dS = |{\mathbb S}^2|=4\pi$.) Next, we substitute \eqref{eq:theta1eq} into the left side of \eqref{eq:theta2eq} to get an equation for $g_{\theta_2}$ that can be integrated directly to obtain $g$; we fix the constant of integration by requiring $g(\theta_1,0)=0$ (which means geometrically that $H$ fixes the prime meridian). One checks easily from \eqref{eq:theta2eq} and \eqref{eq:theta1eq} that the construction gives $g(\theta_1,2\pi)=2\pi$ as required. 

The above construction guarantees $f^\prime>0$ and $g_{\theta_2}>0$ away from the poles, and one can check that the resulting $H$ gives a bi-Lipschitz homeomorphism of the sphere.

\smallskip
\noindent \emph{Remark.} The point of this section is to provide a construction of $H$ that can be implemented in practical examples. Many other homeomorphisms also satisfy the Jacobian condition \eqref{eq:distort}. See Dacorogna and Moser \cite{DM90} for an account of the amazingly varied possibilities. 

\subsection*{Higher dimensions}
In dimensions $4$ and higher, one extends the $3$-dimensional construction by means of generalized spherical coordinates. Induction on the dimension provides the analogue of $g$ on lower dimensional ``latitudinal spheres''. We leave the details to the reader.

\section{\bf Ellipsoidal examples and the geometric factor $G_1$}
\label{sec:ellipsoidex}

The homeomorphism $H : \Sp \to \Sp$ constructed in the preceding section induces a volume-preserving map $T : \Omega \to \B$, as defined in \autoref{sec:setup}. For ellipsoids one could alternatively use the linear map provided by a matrix $M$ with $M(E)=\B$. Which of these two maps will give a better estimate on the eigenvalues in \autoref{Deigen}? That is, which will give a smaller value for the geometric factor $G_1$?

For the linear map one has
\begin{align*}
  T(r\xi)&=M(r\xi) = r M\xi , \\
  R(\xi)&=\frac1{|M\xi|}, \qquad 
  H(\xi)=\frac{M\xi}{|M\xi|}.
\end{align*}
Extending $R$ and $H$ to be homogeneous functions gives 
\[
  R(x)=\frac{|x|}{|Mx|} , \qquad  H(x)=\frac{Mx}{|Mx|} .
\]
Somewhat tedious calculations then show that $G_0=G_1=[V(E)/V(\B)]^{2/d} \lVert M \rVert_{HS}^2 /d$. Thus for ellipsoids we recover our earlier results about linear transformations \cite{LS11b,LS11c}. (Those papers treat more general domains than just ellipsoids, of course.)

Let us now compare this ``linear'' map $T$ with the map $T$ constructed from ``spherical coordinates'' as in the previous section. \autoref{tab1} shows the values of $G_1$ in the linear case, and also shows numerical values from the spherical coordinates construction, for various choices of north pole. Notice in the table that if the ellipsoid has two equal semi-axes (so that it is a body of revolution), and if we take the remaining axis as the north pole, then we obtain the same value for $G_1$ as in the linear case. That equality no longer holds for a generic ellipsoid with unequal semi-axes, and in general the linear construction gives better results than the spherical coordinates one does. 

These observations provide some guidance as to how to choose the north pole when constructing $H$ by the spherical coordinates method, for an arbitrary starlike domain. 

\begin{table}[t]
\begin{center}
  \begin{tabular}{ccccc}
  \toprule
semiaxes (a,b,c) & linear: $G_1=G_0$ & North a & North b & North c\\
\midrule
(1,1,1)& 1 & 1 & 1 & 1\\
(1,1,2)& 1.19055 & 1.24002 & 1.24002 & 1.19055\\
(1,2,2)& 1.25992 & 1.25992 & 1.32057 & 1.32057\\
(1,2,3)& 1.49810 & 1.51620 & 1.73826 & 1.53697\\
\bottomrule
  \end{tabular}
\end{center}
 \vspace*{10pt}
\caption{Values of $G_1$ for the ``linear'' construction of $T$, and for various choices of north pole in the ``spherical coordinates'' construction of $T$.}
  \label{tab1}
\end{table}

\section{\bf Sloshing problem}
\label{sec:sloshing}

We finish the paper by transferring our results to ``sloshing eigenvalues''. On a cylinder $C = \Omega\times[-L,0]$, we consider the following two eigenvalue problems:
\begin{align*}
  \Delta u = 0 & \quad \text{in $C$,}\\
  u=0 & \quad \text{on $\partial C\setminus \Omega$,} \\
  \frac{\partial u}{\partial n}=\widetilde{\lambda} u & \quad \text{on $\Omega$,}
\end{align*}
and
\begin{align*}
  \Delta u &= 0 \quad \text{in $C$,}\\
  \frac{\partial u}{\partial n}&=0 \quad \text{on $\partial C\setminus \Omega$,} \\
  \frac{\partial u}{\partial n}&=\widetilde{\mu} u \quad \text{on $\Omega$.}
\end{align*}
See \autoref{fig:sloshing}. The second eigenvalue problem describes frequencies of sloshing of a fluid in the special case of a cylindrical ``glass'' with uniform cross-sections. (See \cite{FK83} for a historical review and \cite{BKPS,Ib05,KKw,KKu09,KKu11} for recent developments.) 

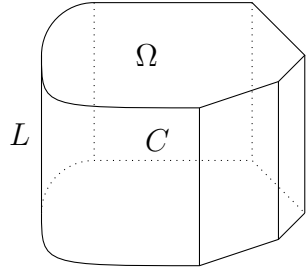
\begin{figure}
\begin{center}
\begin{tikzpicture}[scale=0.7]
  \draw (-1,0) arc (180:90:1) -- (3,1) -- (4,0) -- (3.5,-0.5) -- (2,-1) .. controls (-1,-1) .. (-1,0);
  \draw[yshift=-3cm, dotted] (-1,0) arc (180:90:1) -- (3,1) -- (4,0);
  \draw[yshift=-3cm] (4,0) -- (3.5,-0.5) -- (2,-1) .. controls (-1,-1) .. (-1,0);
  \draw (-1,0) -- (-1,-3) node [left,pos=0.5] {$L$};
  \draw[dotted] (0,1) -- (0,-2);
  \draw[dotted] (3,1) -- (3,-2);
  \draw (4,0) -- (4,-3);
  \draw (3.5,-0.5) -- (3.5,-3.5);
  \draw (2,-1) -- (2,-4);
  \draw (1.2,-1.2) node [below] {$C$};
  \draw (1,0) node {$\Omega$};
%
\end{tikzpicture}
\end{center}
\caption{A sloshing cylinder $C$ with height $L$ and starlike cross-section $\Omega$.
}
\label{fig:sloshing}
\end{figure}

When we want to emphasize the dependence on the cylinder depth, we write the eigenvalues as $\widetilde{\lambda}_j(L)$ and $\widetilde{\mu}_j(L)$. The eigenvalues are determined by separating the vertical and horizontal variables. One finds 
\[
  \widetilde{\lambda}_j(L)=\sqrt{\lambda_j}\coth(\sqrt{\lambda_j}L), \qquad 
  \widetilde{\mu}_j(L)=\sqrt{\mu_j}\tanh(\sqrt{\mu_j}L) ,
\]
where $\lambda_j$ and $\mu_j$ are the Dirichlet  and Neumann eigenvalues of $\Omega$. 
The functions $\Phi_D(a) = \sqrt{a}\coth(\sqrt{a} L)$ and $\Phi_N(a) = \sqrt{a}\tanh(\sqrt{a} L)$ are concave increasing, and so majorization (\autoref{sec:major}) extends our theorems on Dirichlet and Neumann eigenvalues of the Laplacian to sloshing eigenvalues. The Neumann sloshing conclusion, in $2$ dimensions, is that the normalized eigenvalue sum
\[
\sum_{j=2}^n \widetilde{\mu}_j \Big(L\sqrt{\frac{A}{G}} \Big) \sqrt{\frac{A}{G}} = \sum_{j=2}^n \Phi_N \big( \mu_j A/G \big)
\]
is maximal for the disk, for each $L>0$. (Notice here that the depth $L\sqrt{A/G}$ of the cylinder depends on the area and geometric factor of the cross-section.) One can then extend to more general functionals of the $\widetilde{\mu}_j$, by performing a second majorization. 

These methods handle only sloshing in \emph{cylindrical} glasses, although one can use domain monotonicity (see \cite{BKPS}) to obtain bounds for some other shapes of glass. It would be interesting to prove sharp bounds on eigenvalues of \emph{non}-cylindrical sloshing regions, by comparing somehow with a domain having rotational symmetry about the vertical axis, as in \cite{KKw}. 

\section*{Acknowledgments}
This work was partially supported by a grant from the Simons Foundation (\#204296 to Richard Laugesen), and travel funding from the University of Oregon.

We thank Michiel van den Berg for asking about perturbations of the ball and alerting us to his recent work \cite{vdB12}. Our \autoref{perturb2} and \autoref{perturbhigher} resulted from his question. Thanks go also to Julie Clutterbuck, for suggesting we consider nonconstant Robin parameters in \autoref{Reigen}. We are grateful to CIRM--Luminy and MFO-Oberwolfach for funding the stimulating workshops on ``Shape Optimization Problems and Spectral Theory" (May 2012) and ``Geometric Aspects of Spectral Theory'' (July 2012), respectively, during which these conversations took place.

Lorenzo Brasco pointed out how to deduce the estimate on the second eigenvalue (in \autoref{Deigen}) from  Ashbaugh and Benguria's sharp PPW inequality. The de Giorgi Center at the Scuola Normale in Pisa generously supported our participation in the meeting on ``New Trends in Shape Optimization'' (July 2012), at which this paper was completed. Mark Ashbaugh has our thanks for correcting several misstatements about the Robin eigenvalues.

\appendix

\section{\bf Majorization}
\label{sec:major}

To extend from eigenvalue sums to sums of concave functions of eigenvalues we use:
\begin{proposition} \label{pr:major}
Assume $\{ a_j \}$ and $\{ b_j \}$ are increasing sequences of positive real numbers. Then the following statements are equivalent:

(i) $\sum_{j=1}^n a_j \leq \sum_{j=1}^n b_j$ for each $n \geq 1$.

(ii) $\sum_{j=1}^n \Phi(a_j) \leq \sum_{j=1}^n \Phi(b_j)$ for each $n \geq 1$ and all concave increasing functions $\Phi : \R_+ \to \R$.
\end{proposition}
The result is due to Hardy, Littlewood and P\'{o}lya \cite[{\S}3.17]{HLP88}. They treated decreasing sequences $\{ a_j \}$ and $\{ b_j \}$ and a convex increasing function $\Phi$, which is equivalent to \autoref{pr:major} after replacing $\Phi(a)$ with $-\Phi(-a)$. A comprehensive account of majorization methods can be found in the monograph of Marshall, Olkin and Ingram \cite{MO11}. For equality statements in (i) and (ii), including the infinite series case $n=\infty$, see a paper by Laugesen and Morpurgo \cite[Proposition 10]{LM98}.

\section{\bf Orbital and spatial averages}
\label{app:averages}

Equality of orbital and spatial averages on the sphere was needed several times in the paper.
\begin{lemma} \label{le:haar}
\[
\int_{O(d)} f(U\zeta) \, d\gamma(U) = \frac{1}{|\Sp|} \int_\Sp f(\zeta^\prime) \, dS(\zeta^\prime)
\]
for any $f \in L^1(\Sp)$ and each $\zeta \in \Sp$.
\end{lemma}
\begin{proof}[Proof of \autoref{le:haar}]
The right side of the formula equals $\frac{1}{|\Sp|} \int_\Sp f(U\zeta^\prime) \, dS(\zeta^\prime)$ by a change of variable, for each $U$. Integrating with respect to $U$ gives (by Fubini) that
\[
 \frac{1}{|\Sp|} \int_\Sp f(\zeta^\prime) \, dS(\zeta^\prime) = \frac{1}{|\Sp|} \int_\Sp \int_{O(d)} f(U\zeta^\prime) \, d\gamma(U) dS(\zeta^\prime) .
\]
For each $\zeta^\prime$ we change variable with $U \mapsto UV$, where $V$ is chosen so that $V\zeta^\prime=\zeta$. The lemma follows.
\end{proof}

\newcommand{\doi}[1]{%
 \href{http://dx.doi.org/#1}{doi:#1}}
\newcommand{\arxiv}[1]{%
 \href{http://front.math.ucdavis.edu/#1}{ArXiv:#1}}
\newcommand{\mref}[1]{%
\href{http://www.ams.org/mathscinet-getitem?mr=#1}{#1}}

\end{document}